\documentclass[smallextended]{svjour3}

\smartqed
\usepackage{amssymb,amsxtra}
\usepackage{dsfont,latexsym,enumerate}
\usepackage{url}
\usepackage{mathrsfs}

\usepackage{color}
\definecolor{citation}{rgb}{0.2,0.6,0.2}
\definecolor{formula}{rgb}{0.1,0.2,0.6}
\definecolor{url}{rgb}{0,0,0.4}
\usepackage[colorlinks=true,linkcolor=formula,urlcolor=url,citecolor=citation]{hyperref}

\newlength{\defbaselineskip}
\newcommand{\setlinespacing}[1]
           {\setlength{\baselineskip}{#1 \defbaselineskip}}

\usepackage{amsthm}

\numberwithin{equation}{section}
\allowdisplaybreaks[4]

\def\vs{\vspace{1mm}}
\def\dd{d_{\textrm{o}}}

\allowdisplaybreaks
\makeatletter
\DeclareRobustCommand*{\bfseries}{%
	\not@math@alphabet\bfseries\mathbf
	\fontseries\bfdefault\selectfont
	\boldmath
}

\definecolor{forestgreen(web)}{rgb}{0.13, 0.55, 0.13}
\definecolor{darkpowderblue}{rgb}{0.0, 0.2, 0.6}
\definecolor{indigo(dye)}{rgb}{0.0, 0.25, 0.42}
\definecolor{byzantium}{rgb}{0.44, 0.16, 0.39}
\definecolor{cornellred}{rgb}{0.7, 0.11, 0.11}

\def\p{{\phi}}
\def \e{{\alpha}}
\def \r{{\mathds{R}}}
\def \h{{\mathds{H}}}
\def \l{{\mathcal{L}}}
\def \T{{\textup{Tail}}}
\def\Xint#1{\mathchoice
	{\XXint\displaystyle\textstyle{#1}}%
	{\XXint\textstyle\scriptstyle{#1}}%
	{\XXint\scriptstyle\scriptscriptstyle{#1}}%
	{\XXint\scriptscriptstyle\scriptscriptstyle{#1}} %
	\!\int}
\def\XXint#1#2#3{{\setbox0=\hbox{$#1{#2#3}{\int}$}
		\vcenter{\hbox{$#2#3$}}\kern-.5\wd0}}

\def\dashint{\Xint-}

\begin{document}

 \title{Nonlocal Harnack inequalities in the Heisenberg group
	\thanks{The authors are members of~\,``Gruppo Nazionale per l'Analisi Matematica, la Probabilit\`a e le loro Applicazioni (GNAMPA)'' of Istituto Nazionale di Alta Matematica (INdAM).  The first author is supported by the University of Parma via the project ``Regularity, Nonlinear Potential Theory and related topics''.}}

    \author{Giampiero Palatucci \and Mirco Piccinini}

\institute{G. Palatucci \at 
          Dipartimento di Scienze Matematiche, Fisiche e Informatiche,\\
          Universit\`a di Parma\\
          Parco Area delle Scienze 53/a, Campus, 43124 Parma, Italy\\
          \email{giampiero.palatucci@unipr.it}
         \and M. Piccinini \at 
         Dipartimento di Scienze Matematiche, Fisiche e Informatiche,\\
          Universit\`a di Parma\\
          Parco Area delle Scienze 53/a, Campus, 43124 Parma, Italy\\\email{mirco.piccinini@unipr.it}
          }

\titlerunning{Nonlocal Harnack inequailities in the Heisenberg group}
\maketitle	
	\begin{abstract}
		We deal with a wide class of nonlinear integro-differential problems in the Heisenberg-Weyl group~$\mathds{H}^n$, whose prototype is the Dirichlet problem for the $p$-fractional subLaplace equation. 
		These problems arise in many different contexts in quantum mechanics, in ferromagnetic analysis, in phase transition problems, in image segmentations models, and so on, when non-Euclidean geometry frameworks and nonlocal long-range interactions do naturally occur.
		\\ We  prove general Harnack inequalities for the related weak solutions. Also,
		in the case when the growth exponent is $p=2$, we  investigate the asymptotic behavior of the fractional subLaplacian operator, and the robustness of the aforementioned Harnack estimates as the differentiability exponent $s$ goes to $1$.

	\keywords{Harnack inequalities \and fractional Sobolev spaces \and  H\"older continuity \and Heisenberg group \and fractional sublaplacian\vspace{1mm}}
 
	\subclass{35B10 \and 35B45 \and 35B05 \and 35H05 \and  35R05 \and 47G20\vspace{1mm}} 
	\end{abstract}

	\setcounter{equation}{0}\setcounter{theorem}{0}
	
\tableofcontents

	%
	%
	
	\vspace{2mm}

	\setcounter{tocdepth}{2} 
	
	\setlinespacing{1.03}

    \section{Introduction}
    We deal with a very general class of nonlinear nonlocal operators, which include, as a particular case, the fractional subLaplacian. Precisely, let 
     $\Omega$ be a bounded domain in the Heisenberg-Weyl group~$\h^n$, and let $g$ be  in the fractional Sobolev space $W^{s,p}(\h^n)$, for any $s \in (0,\,1)$ and any $p>1$. We shall prove general Harnack inequalities for the weak solutions 
     to the following  class of nonlinear integro-differential problems, 
      \begin{eqnarray}\label{problema}
           \begin{cases}
     \l u = f  & \text{in} \ \Omega,\\[0.4ex]
     u = g  & \text{in}\ \h^n \smallsetminus \Omega,
          \end{cases}
               \end{eqnarray}
     where $f=f(\cdot,u)$ belongs to $L^\infty_{\text{loc}}(\h^n)$ uniformly in $\Omega$, and
     $\l$ is the operator defined by
     \begin{equation}
     	\label{operatore}
     	\l u (\xi)\, =\, P.~\!V. \int_{\h^n}\frac{|u(\xi)-u(\eta)|^{p-2}\big(u(\xi)-u(\eta)\big)}{\dd(\eta^{-1}\circ \xi)^{Q+sp}}\,{\rm d}\eta, \qquad \xi \in \h^n,
     \end{equation}
     with $\dd$ being a homogeneous norm on $\h^n$, and $Q=2n+2$ the usual homogeneous dimension of $\h^n$. The symbol $P.~\!V.$ in the display above stands for ``in the principal value sense''. We immediately refer the reader to Section~\ref{sec_preliminaries} below for the precise definitions of the involved quantities and related properties, as well as for further observations in order to relax some of the assumptions listed in the present section.
     \vspace{2mm}
     
     Integral-differential operators in the form as in~\eqref{operatore} do arise as a generalization of the fractional subLaplacian on the Heisenberg group, naturally defined in the fractional Sobolev space~$H^s(\h^n)$ for any~$s \in (0,\,1)$ as follows
     \begin{equation}\label{fractional_sublaplacian}
      (-\Delta_{\h^n})^su(\xi)\, :=\, C(n,s) \,P.~\!V. \int_{\h^n}\frac{u(\xi)-u(\eta)}{|\eta^{-1}\circ \xi|_{\h^n}^{Q+2s}}\,{\rm d}\eta, \qquad \xi \in \h^n,
     \end{equation}
     where $|\cdot|_{\h^n}$ is the standard homogeneous norm of $\h^n$, 
      and~$C(n,s)$ is a positive constant which depends only on~$n$ and~$s$. In this fashion, the prototype of the wide class of 
   problems in~\eqref{problema} reads as follows,
     \begin{eqnarray} \label{problemino}
           \begin{cases}
     (-\Delta_{\h^n})^s u = 0  & \text{in} \ \Omega,\\[0.4ex]
     u = g  & \text{in}\ \h^n \smallsetminus \Omega,
          \end{cases}
  \end{eqnarray}
      In the last decades, a great attention has been focused on the study of problems involving fractional equations, both from a pure mathematical point of view and for concrete applications since they naturally arise in many different contexts. Despite its relatively short history, the literature is really too wide to attempt any comprehensive treatment in a single paper; we refer for instance to the paper~\cite{DPV12} for an elementary introduction to fractional Sobolev spaces and for a quite extensive (but still far from exhaustive) list of related references.
%
 For what concerns specifically the family of equations in~\eqref{problema} and the corresponding energy functionals, both in the nonlocal and in the local framework, the link with several concrete models arises from many different contexts in Probability (e.~\!g., in non-Markovian coupling for Brownian motions~\cite{BGM18}), in Physics (e.~\!g., in group theory in quantum mechanics~\cite{Wey50}, in ferromagnetic trajectories~\cite{OS22}, in image segmentation models~\cite{CMS10}, in phase transition problems described by Ising models~\cite{PS99}, and many others),
 where the analysis in sub-Riemannian geometry revealed to be decisive. In this respect,  as proven in the literature, Harnack-type inequalities constitute a fundamental tool of investigation.

    \vspace{2mm}

    	Let us focus now merely on regularity and related results in {\it  the fractional panorama in the Heisenberg group}. It is firstly worth stressing that one can find various definitions of the involved operator and related extremely different approaches. In the linear case when $p=2$, an explicit integral definition can be found in the relevant paper~\cite{RT16}, where several Hardy inequalities for the conformally invariant fractional powers of the sublaplacian are proven, and~\cite{CCR15} for related Hardy and uncertainty inequalities on general stratified Lie groups involving fractional powers of the Laplacian; we also refer to~\cite{AM18}, where, amongst other important results, Morrey and Sobolev-type embeddings are derived for fractional order Sobolev spaces. 
	Still in the linear case when~$p=2$, very relevant results have been obtained based on the construction of fractional operators via  a Dirichlet-to-Neumann map associated to degenerate elliptic equations, as firstly seen for the Euclidean framework in the celebrated Caffarelli-Silvestre $s$-harmonic extension. 
	 For this, we would like to mention the very general approach in~\cite{GT21}; the Liouville-type theorem in~\cite{CT16}; the Harnack and H\"older results in Carnot groups in~\cite{FF15}; the connection with the fractional perimeters of sets in Carnot group in~\cite{FMPPS18}.

	For what concerns the more general situation as in~\eqref{operatore} when a $p$-growth exponent is considered, in our knowledge, a regularity theory is very far from be complete; nonetheless, very interesting estimates have been recently proven, as, e.~\!g., in~\cite{KS18,WD20}, and in our recent paper~\cite{MPPP21} where local boundedness and H\"older estimates have been proven for the weak solutions to~\eqref{problema}.

     \vspace{2mm}
     
     In order to state our main results, we need to introduce a special quantity 
     which  plays a central role when dealing with nonlocal operators. Namely, we define the {\em nonlocal tail}\,~$ \T(u;\xi_0,R)$ of a function $u$ centered in $\xi_0\in\h^n$ of radius $R>0$,
     \begin{equation}\label{tail}
      \T(u;\xi_0,R):= \left( R^{sp}\int_{\h^n \smallsetminus B_R(\xi_0)} |u(\eta)|^{p-1}|\eta^{-1}\circ \xi_0|_{\h^n}^{-Q-sp}\,{\rm d} \eta\right)^\frac{1}{p-1}.
     \end{equation}
    We immediately notice that the  quantity above is finite whenever $u \in L^{q}(\h^n)$, with $ q \geq p-1$. The nonlocal tail in~\eqref{tail} can be seen as the natural generalization in the Heisenberg setting of that originally introduced in~\cite{DKP16,DKP14}, and subsequently 
	revealed to be decisive in the analysis of many different nonlocal problems when a fine quantitative control of the 
	long-range interactions is needed;  see for instance the subsequent results proven in~\cite{KMS15,KKP16,IMS16,KKP17} and the references therein.
\vspace{2mm}	

	Our main result reads as follows,
    \begin{theorem}[{\bfseries Nonlocal Harnack inequality}]\label{thm_harnack}
	For any $s \in (0,1)$ and any $p \in (1,\infty)$, let $u \in W^{s,p}(\h^n)$ be a weak solution to~\eqref{problema} such that $u \geq 0$ in $B_R \equiv B_R(\xi_0) \subset \Omega$. Then, for any $B_r$ such that $B_{6r}\subset B_R$, it holds
	\begin{equation}\label{eq_harnack}
		\sup_{B_{r}}u 
		\, \leq \,
		\textbf{c}\, \inf_{B_r}u + \textbf{c} \left(\frac{r}{R}\right)^\frac{sp}{p-1}\T(u_-;\xi_0,R) +\textbf{c}\,r^\frac{sp}{p-1}\|f\|^\frac{1}{p-1}_{L^\infty(B_R)}\,,\\[1.3ex]
	\end{equation}
	where~$\T(\cdot)$ is defined in~\eqref{tail}, $u_-:=\max\{-u,0\}$ is the negative part of the function~$u$, and $\textbf{c}$ depends only on $n$, $s$, $p$, and the structural constant~$\Lambda$ defined in~\eqref{def_lambda}.
    \end{theorem}
    Notice that in the case when $u$ is nonnegative in the whole~$\h^n$, the formulation in~\eqref{eq_harnack} does reduce to that of the classical Harnack inequality.
    \vspace{2mm}
    
    In the particular situation when $u$ is merely a weak supersolution to Problem \eqref{problema}, still in analogy with the classical case when~$s=1$, a weak Harnack inequality can be proven, as stated in the following
    \begin{theorem}[{\bfseries Nonlocal weak Harnack inequality}]\label{thm_weak}
	For any $s \in (0,1)$ and any $p \in (1,\infty)$, let $u \in W^{s,p}(\h^n)$ be a weak supersolution to~\eqref{problema} such that $u \geq 0$ in $B_R \subseteq \Omega$.
	  Then, for any $B_r$ such that $B_{6r}\subset B_R$, it holds
\begin{equation}\label{eq_weakharnack}
 	\left( \ \dashint_{B_{r}}  u^t\,{\rm d}\xi\,\right)^\frac{1}{t} \ \leq \ \textbf{c}\,\inf_{B_{\frac{3}{2}r}} u 
\, +\textbf{c} \left(\frac{r}{R}\right)^\frac{sp}{p-1} \T(u_-;\xi_0,R) \, + \textbf{c}\chi,
\end{equation}
where
$$
\chi = 
\begin{cases}
r^\frac{Qsp}{t(Q-sp)}\|f\|_{L^\infty(B_R)}^\frac{Q}{t(Q-sp)} \qquad \text{for} \ t<\frac{Q(p-1)}{Q-sp} &\quad  \text{if} \ sp<Q,\\[0.8ex]
r^\frac{Q(s-\varepsilon)}{t\varepsilon}\|f\|_{L^\infty(B_R)}^\frac{s}{t\varepsilon}  \qquad \text{for any} \ s-{Q}/{p}< \varepsilon<s \ \text{and}\  t<\frac{(p-1)s}{\varepsilon} &\quad  \text{if} \ sp\geq Q,
\end{cases}
$$
	$\T(\cdot)$ is defined in~\eqref{tail}, and $u_-:=\max\{-u,0\}$ is the negative part of the function~$u$. The constant~$\textbf{c}$ depends only on $n$, $s$, $p$, and the structural constant~$\Lambda$ defined in~\eqref{def_lambda}.
    \end{theorem}
As expected, a nonlocal tail contribution should be still taken into account, and, again, such a contribution in~\eqref{eq_weakharnack} will disappear in the case when the function~$u$ is nonnegative in the whole~$\h^n$. 
   \vspace{3mm}

It is now worth noticing that the main difficulty into the treatment of the equation in~\eqref{problema} lies in the very definition of the leading operator~$\l$ defined in~\eqref{operatore}, which combines the typical issues given by its {\it nonlocal} feature together with the ones given by its {\it nonlinear} growth behaviour and with those naturally arising from the non-Euclidean geometrical structure. 
\vspace{1mm}

For this, some very important tools recently introduced in the nonlocal theory and successfully applied in the fractional sublaplacian on the Heisenberg group, as the aforementioned Caffarelli-Silvestre $s$-harmonic extension, 
and the approach via Fourier representation, as well as
	other  successful tools, like for instance 
	the commutator estimates in~\cite{Sch16}, the pseudo-differential commutator compactness in~\cite{PP14}, and many others, seem not to be adaptable to the framework we are dealing with.
	However, even in such a  nonlinear non-Euclidean framework, we will be able to extend part of the strategy developed in~\cite{DKP14} where nonlocal Harnack inequalities have been proven for the homogeneous version of the analogue of problem~\eqref{problema} in the Euclidean framework.  Further efforts are also needed due to the presence of the non homogeneous datum~$f$, as well as in order to deal with the limit case when $sp=Q$, both of them are novelty even with respect to the results proven in the Euclidean framework in~\cite{DKP14}.
	
\vspace{2mm}

Let now focus on the linear case when $p=2$ when the datum $f$ in the right-hand side of the equation in~\eqref{problema} is zero. It is worth mentioning that the necessity of the presence of the tail term in the Harnack inequalities stated above is a very recent achievement. Indeed, during the last decades, the validity of the classical Harnack inequality without extra positivity assumptions on the solutions has been an open problem in the nonlocal setting, and more in general for integro-differential operators of the form in~\eqref{problema} even in the Euclidean framework. An answer has been eventually given by Kassmann in his breakthrough papers~\cite{Kas07,Kas11}, where a simple counter-example is provided in order to show that positivity cannot be dropped nor relaxed even in the most simple case when $\l$ does coincide with the fractional Laplacian operator $(-\Delta)^s$; see Theorem~1.2 in~\cite{Kas07}. The same author proposed a new formulation of the Harnack inequality without requiring the additional positivity on solutions by adding an extra term, basically a natural tail-type contribution on the right-hand side, in accordance with the result presented here; see Theorem 3.1 in~\cite{Kas11}, where the robustness of the estimates as $s$ goes to $1$ is also presented.
In the same spirit, we also investigate the special linear case in which~$\l$ does reduce to the pure fractional subLaplacian, namely problem~\eqref{problemino}. Firstly, as expected, we prove that the fractional subLaplacian~$(-\Delta_{\h^n})^s$ effectively converges to the standard subLaplacian~$-\Delta_{\h^n}$ as $s$ goes to $1$, as stated in Proposition~\ref{limit_of_fraclap} below. For this, we shall carefully estimate the weighted second order integral form of the fractional subLaplacian with the aim of a suitable Mac-Laurin-type expansion in the Heisenberg group.
       \begin{prop}\label{limit_of_fraclap}
	For any $u \in C^\infty_0(\h^n)$ the following statement holds true
	\begin{equation}\label{limit}
		\lim_{s \rightarrow 1^-}(-\Delta_{\h^n})^s u = -\Delta_{\h^n} u,
	\end{equation}
	where $(-\Delta_{\h^n})^s$ is defined by~\eqref{fractional_sublaplacian}, and $\Delta_{\h^n}$ is the classical subLaplacian in~$\h^n$.
	\end{prop}
Notice that a proof of the result above for fractional subLaplacian on Carnot groups can be found in the relevant paper~\cite{FMPPS18}, via  heat kernel characterization.
\vspace{2mm}

Secondly, we revisit the proofs of Theorems~\ref{thm_harnack}-\ref{thm_weak} by taking care of the dependance of the differentiability exponent~$s$ in all the estimates, so that we are eventually able to obtain the results below in clear accordance with the analogous ones in the Euclidean framework~\cite{Kas11}, by proving that the nonlocal tail term will vanish when $s$~goes to~$1$, in turn recovering the classical Harnack formulation.
   \begin{theorem}\label{thm_harnack2}
   	For any $s \in (0,1)$ let $u \in H^s(\h^n)$ be a weak solution to~\eqref{problemino} such that $u \geq 0$ in $ B_R(\xi_0) \subset \Omega$. Then, the following estimate holds true for any $B_{r}$ such that $B_{6r} \subset B_R$,
   	\begin{equation}
   		\sup_{B_{r}}u \leq \textbf{c}\, \inf_{B_{r}}u + \textbf{c}\, (1-s)\left(\frac{r}{R}\right)^{2s}\T(u_-;\xi_0,R),
   	\end{equation}
   	where~$\T(\cdot)$ is defined in~\eqref{tail} by taking $p=2$ there, $u_-:=\max\{-u,0\}$ is the negative part of the function~$u$, and $\textbf{c}=\textbf{c}\,(n,s)$.
   \end{theorem}
       
         \begin{theorem}\label{thm_weakharnack2}
    	For any $s \in (0,1)$ let $u \in H^s(\h^n)$ be a weak supersolution to~\eqref{problemino}, such that $u \geq 0$ in $B_R(\xi_0) \subseteq \Omega$. Then,  the following estimate holds true for any $B_r$ such that $B_{6r} \subset B_R$, and any $t<Q/(Q-2s)$,
    	\begin{equation}
    		\left( \, \dashint_{B_{r}}  u^t{\rm \, d}\xi\right)^\frac{1}{t} \leq  \textbf{c}\,\inf_{B_{\frac{3r}{2}}} u +\textbf{c} \,(1-s)\left(\frac{r}{R}\right)^{2s} \T(u_-;\xi_0,R),
    	\end{equation}
    	where~$\T(\cdot)$ is defined in~\eqref{tail} by taking $p=2$ there, $u_-:=\max\{-u,0\}$ is the negative part of the function~$u$, and~$\textbf{c}=\textbf{c}\,(n,s)$.
    \end{theorem}

	\vspace{3mm}

	{\it Further developments}. \,Starting from the results proven in the present paper, several questions naturally arise.
	\\*[0.3ex]
	\indent $\bullet$  Firstly, it is worth remarking that we treat general weak solutions, namely by truncation and dealing with the resulting error term as a right hand-side, in the same flavour of the papers~\cite{DKP14,DKP16,KKP17}, in the spirit of De~Giorgi-Nash-Moser. However, one could approach the same family of problems by focusing solely to bounded viscosity solutions in the spirit of Krylov-Safonov. 
	\vspace{1mm}
	
	
	
	$\bullet$ Still in clear accordance with the Euclidean counterpart, one would expect self-improving properties of the solutions to~\eqref{problema}. For this, one should extend the recent nonlocal Gehring-type theorems proven in~\cite{KMS15,Sch16,SM22}. 
	\vspace{1mm}
	
	$\bullet$ One could expect nonlocal H\"arnack inequalities and other regularity results for the solutions to a strictly related class of problems; that is, by adding in~\eqref{problema} a second integral-differential operators, of differentiability exponent~$t>s$ and summability growth~$q>1$, controlled by the zero set of a modulating coefficient: the so-called nonlocal double phase problem, in the same spirit of the Euclidean case treated in~\cite{DFP19,BOS21}, starting from the pioneering results in the local case, when $s=1$, by Colombo and Mingione; see for instance~\cite{DM20,DM20b} and the references therein. 

	 In the same spirit, it could be interesting to understand if our methods do apply in non-Euclidean setting for even more general nonlocal nonstandard growth equations, as the one recently considered in~\cite{CKW22,PT21}.
	\vspace{1mm}


	$\bullet$ 
Recently, mean value properties for solutions to fractional equations have been of great interest. It could be interesting to generalize such an investigation in a fractional non-Euclidean framework as the one considered in the present paper; we refer to the relevant results in~\cite{BDV20,BS21} and the references therein.
	\vspace{1mm}
		
	$\bullet$ Moreover, to our knowledge, nothing is known about Harnack inequalities and more in general about the regularity  for  solutions to  parabolic nonlocal integro-differential equations involving the nonlinear operators in~\eqref{operatore}.
	\vspace{1mm}

$\bullet$ Finally, by starting from the estimates proven in the present paper, in~\cite{Pic22} regularity results up to the boundary have been proven for very general boundary data, and for the related obstacle problem. As expected, a tail contribution naturally appears in those estimates in order to control the nonlocal contributions coming from far. Many subsequent related problems are still open, and not for free because of the possible degeneracy and singularity of~$\l$, as for instance boundary Harnack inequalities or Carleson estimates for the homogeneous case. The boundary H\"older estimates and the comparison results in the aforementioned paper~\cite{Pic22} together with the Harnack estimates presented here could be a starting point for such a delicate investigation. 
Still for what concerns boundary Harnack inequalities, we also refer the reader to~\cite{ROTL21} for  general strategy for equations with possibly unbounded right hand-side data.

%
	
	\vspace{3mm}
	
	{\it To summarize}.\,~The results in the present paper seem to be the first ones
	 concerning Harnack estimates for nonlinear nonlocal equations in the Heisenberg group. We prove that one can extend to the Heisenberg setting the strategy successfully applied in the fractional Euclidean case~(\!\!~\cite{DKP14,DKP16,KMS15}), 
	by attacking even a more general equation which applies to non-zero data and also to the case when $sp\geq Q$.
	 From another point of view, our results can be seen as the (nonlinear) nonlocal extension of the Heisenberg counterpart of the celebrated classical Harnack inequality (\!\!~\cite{AGT17,BLU07,LY13}). Moreover, since we derive all our results for a general class of nonlinear integro-differential operators, via our approach by taking into account all the nonlocal tail contributions in a precise way, we obtain alternative proofs that are new even in the by-now classical case of the pure fractional sublaplacian operator $(-\Delta_{\h^n})^s$; also, in such a case, we are able to prove the robustness of the Harnack estimates with respect to $s$ in the limit as $s$ goes to $1$. 
	 
	 	We would guess that our estimates will be important in a forthcoming nonlinear nonlocal theory in the Heisenberg group.

	\vspace{3mm}
	{\it The paper is organized as follows}.\,~In Section~\ref{sec_preliminaries} below we set up notation, and we briefly recall our underlying geometrical structure, by also recalling the involved functional spaces, and providing a few remarks on the assumptions on the data. A few classical technical tools are also stated. In Section~\ref{sec_recent} we present very recent results for fractional equations in the Heisenberg group.
	In Section~\ref{sec_weak}, we firstly carry out a suitable positivity expansion and some tail estimate. Then we complete the proof of the Harnack inequality with tail, and the weak Harnack inequality with tail, respectively. 
		Section~\ref{sec_limit} is devoted to the asymptotic of the fractional subLaplacian operator, and the robustness of the Harnack inequalities in the linear case.

   
  \vspace{2mm}
   \section{Preliminaries}\label{sec_preliminaries}
In this section we state the general assumptions on the quantity we are dealing with. We keep these assumptions throughout the paper.
Firstly, notice that we will follow the usual convention of denoting by~$\textbf c$ a general positive constant which will not necessarily be the same at different occurrences and which can also change from line to line. For the sake of readability, dependencies of the constants will  be often omitted within the chains of estimates, therefore stated after the estimate. 
  Relevant dependencies on parameters will be emphasized by using parentheses.
  
   \subsection{The Heisenberg-Weyl group}\label{sec_heis}
    We start by very briefly recalling a few well-known facts about the Heisenberg group; see for instance~\cite{BLU07} for a more exhaustive treatment.
\vspace{1mm}

We denote points in~$\r^{2n+1}$ by
   $
   \xi := (z,t) = (x_1,\dots,x_n, y_1,\dots,y_n,t).
   $

   For any~$\xi,\xi'\in \r^{2n+1}$, consider the group multiplication~$\circ$ defined by
   	\begin{eqnarray*}
   	\xi \circ \xi' & := &\left(x+x',\, y+y',\, t+t'+2\langle y,x'\rangle-2\langle x,y'\rangle \right) \\*[0.3ex]
   	 & = & \left( x_1+x_1', ..., x_n+x_n',\, y_1+y_1', ..., y_n+y_n',\, t+t' +2 \sum_{i=1}^n\big(
   	y_ix_i'-x_iy'_i\big) 
   	\right).
   \end{eqnarray*}
%
%
 
   For any~$\lambda>0$, 
 the  automorphism group~$({\Phi}_\lambda)_{\lambda>0}$ on~$\r^{2n+1}$ is defined by
   	                 $\xi   \mapsto {\Phi}_{\lambda}(\xi):=(\lambda x,\, \lambda y,\, \lambda^2 t)$, and, as customary,
   $Q\equiv 2n+2$ is the {\it homogeneous dimension of}~$\r^{2n+1}$ with respect to~$({\Phi}_\lambda)_{\lambda>0}$, 
  so  that the Heisenberg-Weyl group~$\h^n := (\r^{2n+1},\circ, {\Phi}_\lambda)$ is a homogeneous Lie group.
   
   The Jacobian base of the Heisenberg Lie algebra~$\mathfrak{h}^n$ of~$\h^n$ is given by
   $$
   X_j := \partial_{x_j} +2y_j\partial_t, \quad X_{n+j}:= \partial_{y_j}-2x_j\partial_t, \quad 1 \leq j\leq n, \quad T:=\partial_t.
   $$
Since 
   $[X_j,X_{n+j}]=-4\partial_t$ {for every } $1 \leq j \leq n$, it plainly follows that
   \begin{eqnarray*}
   	&& \textup{rank}\Big(\textup{Lie}\{X_1,\dots,X_{2n},T\}(0,0)\Big) 
   	\ = \ 2n+1,
   \end{eqnarray*}
   so that~$\h^n$ is a Carnot group with the following stratification of the algebra
   $$
   \mathfrak{h}^n  = \textup{span}\{X_1,\dots,X_{2n}\} \oplus \textup{span}\{T\}.
   $$
   We have now the following 
   \begin{defn}
   	\label{def_homnorm}
   	A \textup{homogeneous norm} on $\h^n$ is a continuous function {\rm (}with respect to the Euclidean topology\,{\rm )} ${d_{\rm o}} : \h^n \rightarrow [0,+\infty)$ such that:
   	\begin{enumerate}[\rm(i)]
   		\item{
   			${d_{\rm o}}({\Phi}_\lambda(\xi))=\lambda {d_{\rm o}}(\xi)$, for every $\lambda>0$ and every $\xi \in \h^n$;
   		}\vspace{1mm}
   		\item{
   			${d_{\rm o}}(\xi)=0$ if and only if $\xi=0$.}
   	\end{enumerate}
   A homogeneous norm ${ d_{\rm o}}$ is {\rm symmetric} if ${d_{\rm o}}(\xi^{-1})={d_{\rm o}}(\xi)$, for all $\xi \in \h^n$;
   \end{defn}
   	If ${d_{\rm o}}$ is a homogeneous norm on $\h^n$, then the function~$\Psi$ defined by
   	$$
   	\Psi(\xi,\eta):={d_{\rm o}} (\eta^{-1}\circ \xi)
   	$$
   	is a pseudometric on $\h^n$.
   	\vspace{1mm}
   	We recall that {\it the standard homogeneous norm $|\cdot|_{\h^n}$ on $\h^n$} is given by
   	\begin{equation}\label{korany_folland}
   		|\xi|_{\h^n}:= \left(|z|^4+t^2\right)^\frac{1}{4}, \quad \forall \xi :=(z,t) \in \r^{2n+1}.
   	\end{equation}
		For any fixed~$\xi_0 \in \h^n$ and~$R>0$ we denote with~$B_R(\xi_0)$ the ball with center~$\xi_0$ and radius~$R$ defined by
    $$
    B_R(\xi_0) := \left\{ \xi \in \h^n : |\xi_0^{-1}\circ \xi|_{\h^n} < R\right\}.
    $$
    
It is now worth noticing that for any homogeneous norm~$d_{\rm o}$ on $\h^n$ one can prove the existence of a positive constant~$\Lambda$ such that
\begin{equation}\label{def_lambda}
   	\Lambda^{-1}|\xi|_{\h^n}\leq {d_{\rm o}}(\xi) \leq \Lambda|\xi|_{\h^n}, \qquad \forall \xi \in \h^n.
\end{equation}
   As a consequence,
   in most of the estimates in the forthcoming proofs, one can simply take into account the pure homogeneous norm defined in~\eqref{korany_folland} with no modifications at all.  
   \vspace{1mm}

   In the analysis of the special case when the integro-differential operator~\eqref{operatore} does reduce to the standard fractional subLaplacian, we will need to obtain fine estimates by taking into account the differentiability exponent~$s$ near 1, and thus several modifications with respect to the proof of similar estimates in the Euclidean framework are needed. In particular, in order to obtain the desired characterization  of the asymptotic behaviour as $s$ goes to $1$ of the fractional subLaplacian, and consequently proving the consistency of our Harnack estimates with tail in the limit (see Section~\ref{sec_limit}), we are able to overcome some difficulties coming from the non-Euclidean structure considered here by making use of a suitable MacLaurin-type expansion.
%
%
   \begin{defn}
   	Let $u \in C^\infty(\h^n;\r)$.
	Then, for any $m \in \mathbb{N} \cup \{0\}$ there exists a unique polynomial $P$ being ${\Phi}_{\lambda}$-homogeneous of degree at most $m$ such that
   	$$
   	(X_1,\dots,X_{2n},T)^{\beta} P(0)= (X_1,\dots,X_{2n},T)^{\beta} u(0)
   	$$
   	for any multi-index $\beta= (\beta_1,\dots,\beta_{2n+1})$ with $|\beta|_{\h^n}=\beta_1+\cdots+\beta_{2n}+2\beta_{2n+1} \leq m$. We say that $P:=P_m(u,0)(\xi)$ is {\rm``}the {\rm MacLaurin polynomial} of ${\Phi_{\lambda}}$-degree $m$ associated to $u${\rm ''}.
   \end{defn}
   In the case of the Heisenberg group one can  explicitly write the MacLaurin polynomial of ${\Phi}_\lambda$-degree 2; we have
$$	P_2(u,0)(x_1,\dots,x_{2n},t) \, = \, u(0)+ \nabla_{\h^n}u(0)\cdot z + \partial_t u (0) \cdot t 
   	   + \frac{1}{2} \langle x, D^{2,*}_{\h^n}u (0) \cdot x \rangle ,
$$
   where  the subgradient~$\nabla_{\h^n}u$ is given by
   $\nabla_{\h^n} u(\xi) := \left( X_1 u(\xi),\dots, X_{2n}u(\xi)\right)$, and
$D^{2,*}_{\h^n}$ is the symmetrized horizontal Hessian matrix; that is,
   \begin{equation}\label{symm_hess}
   	D^{2,*}_{\h^n} u(\xi) := \left(\frac{1}{2}\big(X_iX_ju(\xi)+X_jX_iu(\xi)\big)\right)_{i,j =1,\dots,2n}
   \end{equation}
   \begin{defn}
   	Let $u \in C^\infty(\h^n;\r)$, $\xi_0 \in \h^n$, and $m \in \mathbb{N} \cup \{0\}$. Let us consider the MacLaurin polynomial $P_m (u(\xi_0 \circ \cdot),0)$ of the function
   	$\xi \longmapsto u(\xi_0 \circ \xi)$
   	 The polynomial 
   	$$
   	P_m(u,\xi_0)(\xi):=P_m(u(\xi_0 \circ \cdot),0)(\xi_0^{-1} \circ \xi),
   	$$
   	is the {\rm Taylor polynomial} of $\h^n$-degree $m$ centered at $\xi_0$ associated to $u$.
   \end{defn}
 One can prove the following
   \begin{prop}{\rm (see for instance~\cite[Corollary~20.3.5]{BLU07})}.\label{prop_taylor}
   	For every $u \in C^{m+1}(\h^n;\r)$, $\xi_0 \in \h^n$ and $m \in \mathbb{N} \cup \{0\}$, we have that
   	$$
   	u(\xi)= P_m(u,\xi_0)(\xi) + o(|\xi_0^{-1}\circ \xi|_{\h^n}^{m+1}).
   	$$
   \end{prop}
   
   \vspace{2mm}
   \subsection{The fractional framework} We now introduce the natural fractional framework;  we point out that our setup is in align with Section~2 in~\cite{MPPP21}.
   \vspace{1mm}
   
   Let~$p>1$ and~$s \in (0,1)$,  the Gagliardo seminorm of a measurable function~$u:\h^n\to\r$ is given by
   \begin{equation}
   	[u]_{W^{s,p}} := \left(\,\int_{\h^n}\int_{\h^n} \frac{|u(\xi)-u(\eta)|^p}{|\eta^{-1}\circ \xi|_{\h^n}^{Q+sp}}{ \rm d}\xi {\rm d}\eta\right)^\frac{1}{p},
   \end{equation}
   and the fractional Sobolev space~$W^{s,p}(\h^n)$ given by
    \begin{equation}
    	W^{s,p}(\h^n) := \Big\{ u \in L^p(\h^n): [u]_{W^{s,p}(\h^n)}<+\infty\Big\},
    \end{equation}
   is equipped with the natural norm
   \begin{equation}
   	\|u\|_{W^{s,p}(\h^n)} := \left(\|u\|_{L^p(\h^n)}^p+[u]_{W^{s,p}}^p\right)^\frac{1}{p}, \qquad \forall u \in W^{s,p}(\h^n).
   \end{equation}
   Similarly, given a domain~$\Omega \subset \h^n$,  the fractional Sobolev space~$W^{s,p}(\Omega)$ is given by   \begin{equation}
   W^{s,p}(\Omega):= \left\{ u \in L^p(\Omega): \left(\,\int_{\Omega}\int_{\Omega} \frac{|u(\xi)-u(\eta)|^p}{|\eta^{-1}\circ \xi|_{\h^n}^{Q+sp}}\,{\rm d}\xi {\rm d}\eta\right)^\frac{1}{p}<+\infty\right\}
   \end{equation}
   endowed with the norm
   \begin{equation}
    	\|u\|_{W^{s,p}(\Omega)} := \left(\|u\|_{L^p(\Omega)}^p+\int_{\Omega}\int_{\Omega} \frac{|u(\xi)-u(\eta)|^p}{|\eta^{-1}\circ \xi|_{\h^n}^{Q+sp}} \,{\rm d}\xi {\rm d}\eta\right)^\frac{1}{p}, \qquad \forall u \in W^{s,p}(\Omega).
   \end{equation}
   
   We denote by $W^{s,p}_0(\Omega)$ the closure of~$C^\infty_0(\Omega)$ in~$W^{s,p}(\h^n)$.
   \vspace{2mm}

   %
   We conclude this section by recalling the natural definition of weak solutions to the class of problem we are deal with; that is,
   \begin{eqnarray}\label{problema1}
 \begin{cases}  \l u = f \quad & \Omega\subset \h^n,\\[1ex]
   u=g \quad & \h^n \smallsetminus \Omega,
\end{cases}
   \end{eqnarray}   
  where the datum~$f\equiv f(\cdot,u) \in L^\infty_{\textrm{loc}}(\h^n)$ uniformly in $\Omega$,
  the nonlocal boundary datum~$g$ belongs to $W^{s,p}(\h^n)$, and the leading operator~$\l$ is an integro-differential operator of differentiability exponent~$s \in (0,1)$ and summability exponent~$p>1$ given by
  $$
  \l u(\xi) = P.~\!V. \int_{\h^n}\frac{|u(\xi)-u(\eta)|^{p-2}\big(u(\xi)-u(\eta)\big)}{d_{\rm o}(\eta^{-1}\circ \xi)^{Q+sp}}\,{\rm d}\eta, \qquad \xi \in \h^n,
  $$
  with~$d_{\rm o}$ being an homogeneous norm according to Definition~\ref{def_homnorm}.
  
For any~$g \in W^{s,p}(\h^n)$,  consider the classes of functions
   $$
    \mathcal{K}^{\pm}_g (\Omega):=\Big\{  v \in W^{s,p}(\h^n): (g-v)_\pm \in W^{s,p}_0(\Omega)\Big\},
   $$
   and 
   $$
   \mathcal{K}_g(\Omega) := \mathcal{K}^+_g(\Omega) \cap \mathcal{K}^-_g(\Omega) = \Big\{  v \in W^{s,p}(\h^n): v-g \in W^{s,p}_0(\Omega)\Big\}.
   $$ 
   We have the following
    \begin{defn}\label{solution}
	A function $u \in \mathcal{K}^-_g(\Omega)$ (respectively, $\mathcal{K}^+_g(\Omega)$) is a {\rm weak subsolution} (respectively, {\rm supersolution}) to~\eqref{problema1} if
     \begin{align*}
	   &\int_{\h^n}\int_{\h^n}\frac{|u(\xi)-u(\eta)|^{p-2}(u(\xi)-u(\eta))(\psi(\xi)-\psi(\eta))}{d_{\rm o}(\eta^{-1}\circ \xi)^{Q+sp}}{ \rm d}\xi {\rm d}\eta\\		
       & \qquad \qquad \qquad \leq (\geq, {\rm resp.}) \int_{\h^n}f(\xi,u(\xi))\psi(\xi){\rm d}\xi,
	\end{align*}
	for any nonnegative  $ \psi \in W_0^{s,p}(\Omega)$.
	\\ A function u is a \textup{weak solution} to problem~\eqref{problema1} if it is both a weak sub- and supersolution.
	 \end{defn}
   
   	\begin{rem}\label{rem_tailspace}
		The requirement on the boundary datum~$g$ to be in the whole~$W^{s,p}(\h^n)$ can be weakened by assuming a local fractional differentiability, namely $g \in W_{\textrm{\rm loc}}^{s,p}(\Omega)$, in addition to the boundedness of its nonlocal tail; i.~\!e., $\text{\rm Tail}(g; \xi_0, R)~<~\infty$, for some $\xi_0\in\h^n$ and some $R>0$.  This is not restrictive, and it does not bring  modifications in the rest of the paper. For further details on the related~``Tail space'', we refer the interested reader to papers{\rm ~\cite{KKP16,KKP17}}. 
	\end{rem}
	
%
  \vspace{2mm}
   \subsection{Classical technical tools}
As in the classical variational approach to local Harnack estimates, the following well-known iteration lemmata are needed.
   \begin{lemma}\label{giusti}
   Let $\beta >0$ and let $\{A_j\}_{j \in \mathbb{N}}$ be a sequence of real positive numbers such that 
 $A_{j+1} \leq \textbf{c}_0 b^j A_j^{1+\beta}$   with $\textbf{c}_0 >0$ and $b>1$.
 \\  If $A_0 \leq \textbf{c}_0^{-\frac{1}{\beta}}b^{-\frac{1}{\beta^2}}$, then we have 
  $A_j \leq b^{-\frac{j}{\beta}}A_0$,
   which in particular yields $\displaystyle\lim_{j \rightarrow \infty}A_j=0$.
   \end{lemma}

   \begin{lemma}\label{giaquinta_giusti}
   Let $g=g(t)$ be a nonnegative bounded function defined for $0 \leq T_0 \leq t \leq T_1$. Suppose that for $T_0 \leq t <\tau \leq T_1$ we have
   \begin{equation*}
   g(t) \leq \textbf{c}_1(\tau-t)^{-\theta}+\textbf{c}_2 +\zeta g(\tau),
   \end{equation*}
   where $\textbf{c}_1,\textbf{c}_2,\theta$ and $\zeta <1$ are nonnegative constant. Then, there exists a constant $\textbf{c}$ depending only on $\theta$ and $\zeta$, such that for every $\rho,R,T_0\leq \rho <R \leq T_1$, we have
   \begin{equation*}
   g(\rho) \leq \textbf{c}\, \big(\textbf{c}_1(R-\rho)^{-\theta}+\textbf{c}_2\big).
   \end{equation*}
   \end{lemma}

\vspace{2mm}

\section{Some recent results on fractional operators in the Heisenberg group}\label{sec_recent}
   Similarly to what happens in the Euclidean case, a fractional Sobolev embedding can be proved in the non-Euclidean setting of the Heisenberg group. Indeed, the following result holds true,
   \begin{theorem}\label{sobolev}
   	Let~$p>1$ and~$s \in (0,1)$ such that~$sp<Q$. For any measurable compactly supported function~$u : \h^n \rightarrow \r$ there exists a positive constant~$c=c(n,p,s)$ such that
   	$$
   	\|u\|^p_{L^{p^*}(\h^n)} \leq c[u]_{W^{s,p}}^p,
   	$$
   	where~$p^*:=\frac{Qp}{Q -sp}$ is the critical Sobolev exponent.
   \end{theorem} 
   For the proof of the previous statement we refer to Theorem~2.5 in~\cite{KS18}, where the authors extend the same strategy used in~\cite{DPV12,PSV13} in order to prove the fractional Sobolev embedding in the Euclidean setting.
   \vspace{2mm}

Before proving the weak Harnack inequality, namely Theorem~\ref{thm_weak}, we also need to recall a boundedness estimate and a Caccioppoli-type one for the weak sub- and super-solution to~\eqref{problema}.
\begin{theorem}[{\bfseries Local boundedness}]{\emph{[Theorem 1.1 in~\cite{MPPP21}].}}\label{thm_bdd}
		Let $s \in (0,1)$ and $p \in (1,\infty)$, let $u \in W^{s,p}(\h^n)$ be a weak subsolution to~\textup{(\ref{problema})}, and let $B_r \equiv B_r(\xi_0)   \subset \Omega$. Then the following estimate holds true, for any $\delta \in (0,1]$,
		\begin{equation}		\label{eq_bdd}
			\sup_{B_{r/2}}u \, \leq\,   \delta \,\textup{Tail}(u_+; \xi_0, r/2) + \frac{\textbf{c}}{\delta^{\gamma}}  \left(\,\dashint_{B_r}u_+^p{\rm d}\,\xi\right )^\frac{1}{p}, 
		\end{equation}
		where $\textup{Tail}(u_+;\xi_0,r/2)$ is defined in~\eqref{tail}, $u_+:=\max\left\{u,\, 0\right\}$ is the positive part of the function $u$,  $\gamma=\frac{(p-1)Q}{sp^2}$, and the constant $\textbf{c}$ depends only on $n,p,s$, $\|f\|_{L^\infty(B_r)}$ and the structural constant $\Lambda$ defined in \eqref{def_lambda}.
	\end{theorem}
	
  \begin{theorem}[{\bfseries Caccioppoli estimates with tail}]{\emph{[Theorem 1.3 in~\cite{MPPP21}].}}\label{s3_lem2}
  Let~$p>1$, 
  $q \in (1,p)$, $d>0$ and let $u \in W^{s,p}(\h^n)$ be a weak supersolution to problem~\eqref{problema} such that $u \geq 0$ in $B_R(\xi_0) \subset \Omega$. Then, for any $B_r \equiv B_r(\xi_0) \subset B_R(\xi_0)$ and any nonnegative $\p \in C^\infty_0(B_{r})$, the following estimate holds true
  \begin{eqnarray}\label{caccioppoli}
 &&  \int_{B_{r}}\int_{B_{r}}|\eta^{-1} \circ \xi|_{\h^n}^{-Q-sp}|w(\xi)\p(\xi)-w(\eta)\p(\eta)|^p \, {\rm d}\xi{\rm d}\eta \notag\\
&&\quad    \leq \textbf{c} \int_{B_{r}}\int_{B_{r}} |\eta^{-1} \circ \xi|_{\h^n}^{-Q-sp}\big(\max\big\{w(\xi),\,w(\eta)\big\}\big)^p|\p(\xi)-\p(\eta)|^p\,{\rm d}\xi {\rm d}\eta\\
&&\qquad     + \ \textbf{c} \left(\sup_{\xi \in \textup{supp} \, \p}\int_{\h^n \smallsetminus B_{r}}|\eta^{-1}\circ \xi|_{\h^n}^{-Q-sp} \, {\rm d}\eta\notag\, +\, d^{1-p}R^{-sp} [\T(u_-;\xi_0,R)]^{p-1}\right)\notag\\
   &&\qquad   \times \int_{B_{r}}\!w^p(\xi)\p^p(\xi) \, {\rm d}\xi + cd^{1-q}r^Q \|f\|_{L^\infty(B_R)}\,,\notag
  \end{eqnarray}
  where $w:=(u+d)^\frac{p-q}{p}$, and the constant~$\textbf{c}$ depends only on $n$, $p$, $q$, $\|\p\|_{L^\infty({\rm supp}\,\p)}$ and the structural constant $\Lambda$ defined in \eqref{def_lambda}.
  \end{theorem}

\vspace{0mm}  
   \section{Proofs of the Harnack estimates with tail}\label{sec_weak}
%
   %
   %
For the sake of readability, from now on we adopt the following notation,
   \begin{equation}\label{d_nu}
   	{\rm d} \nu := d_{\rm o}(\eta^{-1} \circ \xi)^{-Q-sp} \,{\rm d}\xi{\rm d}\eta.
   \end{equation}
     
   \subsection{Expansion of positivity}
   In this section we prove a careful estimate of the weak supersolutions to~\eqref{problema}, by generalizing the original strategy applied in the local framework as well as that in the fractional Euclidean one. Clearly, we have to take into account the needed modifications to handle the difficulties given in the fractional Heisenberg framework; also, the presence of the nonhomogeneous datum~$f$ would require further care when some iterative argument will be called.

   \begin{lemma}\label{s3_lem1}
   Let  with $s \in (0,1)$ and $p \in (1,\infty)$, and let $u\in W^{s,p}(\h^n)$ be a weak supersolution to problem~\eqref{problema} such that $u \geq 0$ in $B_R\equiv B_R(\xi_0)\subset \Omega$. Let $k \geq 0$, and suppose that there exists $\sigma \in (0,1]$ such that
   \begin{equation}\label{s3_lem_hyp}
   \big|B_{6r} \cap \big\{u \geq k\big\}\big| \, \geq\,  \sigma|B_{6r}|,
   \end{equation}
   for some $r>0$ such that $B_{8r}  \subset B_R$. Then there exists a constant $\bar{\textbf{c}} \equiv \bar{\textbf{c}}\big(n,p,s,\|f\|_{L^\infty(B_R)}, \Lambda\big)$ such that
  \begin{equation}\label{s3_lem1_thesis}
  \left|B_{6r} \cap \left\{u \leq 2\delta k - \frac{1}{2}\left(\frac{r}{R}\right)^\frac{sp}{p-1}\T(u_-;\xi_0,R)\right\}\right| 
  \, \leq\,  \frac{\bar{\textbf{c}}}{\sigma \log\left(\frac{1}{2\delta}\right)}|B_{6r}|
  \end{equation}
  holds for all $\delta \in (0,1/4)$, where $\T(\cdot)$ is defined in~\eqref{tail}.
  \end{lemma}

  \begin{proof}
  We firstly notice that with no loss of generality one can suppose $\T(u_-;\xi_0,R)\!\!~>~\!\!0$, otherwise~\eqref{s3_lem1_thesis} plainly follows from~\eqref{s3_lem_hyp} by choosing the constant~$\bar{c}$   large enough.
  	
Take a smooth function~$\phi\in C^\infty_0(B_{7r})$ such that
  $$
  0 \leq \p \leq 1, \quad \p \equiv 1 \quad \mbox{in } B_{6r}, \quad \mbox{and} \quad |\nabla_{\h^n} \p| \leq c/r\,,
  $$
an set
\begin{equation}\label{8.0}
  d := \frac{1}{2}\left(\frac{r}{R}\right)^\frac{sp}{p-1}\T(u_-;\xi_0,R), \quad \mbox{and} \quad \tilde{u}:=u+d\,.
\end{equation}
Now, choose $\psi := \tilde{u}^{1-p}\p^p$ as a test function in Definition~\ref{solution} by making use of the fact that $u$ is a supersolution. It follows
  \begin{eqnarray}\label{s3 lem 1}
   && \int_{\h^n}f(\xi,u)\tilde{u}^{1-p}(\xi)\p^p(\xi){\rm d} \xi  \notag \\*[0.8ex]
   && \qquad\leq\  \int_{B_{8r}}\int_{B_{8r}} |\tilde{u}(\xi)-\tilde{u}(\eta)|^{p-2}\big(\tilde{u}(\xi)-\tilde{u}(\eta)\big)\big(\tilde{u}^{1-p}(\xi)\p^p(\xi)-\tilde{u}^{1-p}(\eta)\p^p(\eta)\big) \,{\rm d}\nu\notag\\*
   & &\qquad\quad +\, \int_{\h^n \smallsetminus  B_{8r}}\int_{B_{8r}}|\tilde{u}(\xi)-\tilde{u}(\eta)|^{p-2}\big(\tilde{u}(\xi)-\tilde{u}(\eta)\big)\tilde{u}^{1-p}(\xi)\p^p(\xi) \,{\rm d}\nu\notag\\*[0.8ex]
   & & \qquad\quad -\, \int_{B_{8r}}\int_{\h^n \smallsetminus B_{8r}}|\tilde{u}(\xi)-\tilde{u}(\eta)|^{p-2}\big(\tilde{u}(\xi)-\tilde{u}(\eta)\big)\tilde{u}^{1-p}(\eta)\p^p(\eta) \,{\rm d}\nu\notag\\*[0.8ex]
   & &\qquad =: \ I_1 + I_2+I_3\,,
  \end{eqnarray} 
  where we have also used the definition of~$\tilde{u}$ and the translation invariant property of the fractional seminorms.
   
We start by estimating the integral contribution in the left-hand side. Thanks to the hypothesis on $f$ and in view of the definition of~$\psi$, we have get
$$
   \int_{\h^n}f(\xi,u)\tilde{u}^{1-p}(\xi)\p^p(\xi)\,{\rm d}\xi 
   \, \geq\,  \int_{B_{6r}}\big(f(\xi,u)\big)_+\tilde{u}^{1-p}(\xi)\,{\rm d} \xi 
   \,-\, cr^{Q}d^{1-p}\|f\|_{L^\infty (B_R)}.
$$
  Now, we focus on the terms on the right-hand side. We can treat the first integral~$I_1$ as in the proof of the Logarithmic Lemma~1.4 in~\cite{MPPP21}, so that
  $$
  I_1 \, \leq\, -\frac{1}{c}\int_{B_{6r}}\int_{B_{6r}} \, \left|\log \left(\frac{\tilde{u}(\xi)}{\tilde{u}(\eta)}\right)\right|^p \,{\rm d}\nu +c r^{Q-sp}\,.  
  $$
  We now proceed to estimating the second integral~$I_2$ in~\eqref{s3 lem 1}, in turn obtaining an estimate for~$I_3$ as well. We split~$I_2$ as follows,
  \begin{eqnarray}\label{9.0}
   I_2 & =&  \int_{\h^n \smallsetminus B_{8r}\cap \{\tilde{u}(\eta)<0\}} \int_{B_{8r}} |\tilde{u}(\xi)-\tilde{u}(\eta)|^{p-2}\big(\tilde{u}(\xi)-\tilde{u}(\eta)\big)\tilde{u}^{1-p}(\xi)\p^p(\xi) \, {\rm d}\nu\notag\\*
   && \ + \ \int_{\h^n \smallsetminus B_{8r}\cap \{\tilde{u}(\eta)\geq 0\}} \int_{B_{8r}} |\tilde{u}(\xi)-\tilde{u}(\eta)|^{p-2}\big(\tilde{u}(\xi)-\tilde{u}(\eta)\big)\tilde{u}^{1-p}(\xi)\p^p(\xi) \, {\rm d}\nu\notag\\*[0.8ex]
    & =:&  I_{2,1}+I_{2,2}.
  \end{eqnarray}
  The contribution in~$I_{2,1}$ can be estimated as follows
  \begin{eqnarray*}
   I_{2,1} & = & \int_{\h^n \smallsetminus B_{8r}} \int_{B_{8r}} |\tilde{u}(\xi)+(\tilde{u}(\eta))_-|^{p-1}\tilde{u}^{1-p}(\xi)\p^p(\xi) \, {\rm d}\nu\\*[0.8ex]
   & \leq &cr^Q \int_{\h^n \smallsetminus B_{8r}} \, \left(1+\frac{(u(\eta))_-}{d}\right)^{p-1} \, |\eta^{-1} \circ \xi_0|_{\h^n}^{-Q-sp} \, {\rm d}\eta \\*[0.8ex]
   &  \leq  &cr^{Q-sp}+cr^Qd^{1-p}R^{-sp}\T(u_-;\xi_0,R)^{p-1}\\*[0.8ex]
   & = & c(n,s,p,\Lambda)r^{Q-sp}\,,
   \end{eqnarray*}
   where we have used that, for any $\xi \in B_{7r}$ and any $\eta \in \h^n \smallsetminus B_{8r}$,
  \begin{eqnarray}\label{9.1}
	\frac{|\eta^{-1} \circ \xi_0|_{\h^n}}{|\eta^{-1} \circ \xi|_{\h^n}} & \leq & \frac{(|\eta^{-1} \circ \xi|_{\h^n}+|\xi \circ \xi_0|_{\h^n})}{|\eta^{-1}\circ \xi|_{\h^n}}\notag\\
	& \leq  & 1 + \frac{7r}{|\eta^{-1}\circ \xi_0|_{\h^n}-|\xi^{-1}\circ \xi_0|_{\h^n}}
	\  \leq\  8.
   \end{eqnarray}
   
   For~$I_{2,2}$, notice that~$\tilde{u} \geq 0$ in $B_{7r}$. Then, $\tilde{u}(\xi)-\tilde{u}(\eta) \leq \tilde{u}(\xi)$, for every  $\eta \in \h^n \smallsetminus B_{8r}\cap \big\{\tilde{u}(\eta)\geq 0\big\}$ and every $\xi \in B_{7r}$; we get
\begin{equation}\label{9.2}
   I_{2,2} \,\leq\, c r^{Q-sp}.
\end{equation}
  Thus, by combining~\eqref{9.0} with \eqref{9.1} and~\eqref{9.2}, we arrive at
   $$
   I_2 + I_3 \, \leq\, cr^{Q-sp}\,,
   $$
   for a constant~$c$ depending only on $n$, $p$, $s$ and $\Lambda$. All in all, 
   \begin{equation*}
    \int_{B_{6r}}\int_{B_{6r}} \, \left|\log \left(\frac{\tilde{u}(\xi)}{\tilde{u}(\eta)}\right)\right|^p \, {\rm d}\nu
     +  \int_{B_{6r}}\big(f(\xi,u)\big)_+\tilde{u}^{1-p}(\xi){\rm d} \xi
     \  \leq\  cr^{Q-sp}+ cr^{Q}d^{1-p}\|f\|_{L^\infty(B_R)}\,.
   \end{equation*}  
   
   Now, recalling the particular choice of $d$ in~\eqref{8.0}, one can deduce the  existence of a constant $c:=c(n,p,s,\|f\|_{L^\infty(B_R)},\Lambda)$ such that 
   \begin{equation}\label{s3 lem 2}
   	 \int_{B_{6r}}\int_{B_{6r}} \, \left|\log \left(\frac{\tilde{u}(\xi)}{\tilde{u}(\eta)}\right)\right|^p \, {\rm d}\nu + \int_{B_{6r}}(f(\xi,u))_+\tilde{u}^{1-p}(\xi) \, {\rm d} \xi 
	 \ \leq \ cr^{Q-sp}.
   \end{equation}
   
   For any $\delta \in (0,1/4)$, define
   $$
    v := \left(\min \left\{\log \frac{1}{2 \delta},\, \log \frac{k+d}{\tilde{u}}\right\}\right)_+.
   $$
    Since $v$ is a truncation of $\log(k+d)-\log(\tilde{u})$, the estimate in~\eqref{s3 lem 2} yields
   $$
   \int_{B_{6r}}\int_{B_{6r}}|v(\xi)-v(\eta)|^p \, {\rm d}\nu \leq \int_{B_{6r}}\int_{B_{6r}} \, \left|\log \left(\frac{\tilde{u}(\xi)}{\tilde{u}(\eta)}\right)\right|^p \, {\rm d}\nu
   \  \leq\ cr^{Q-sp}
   $$
    By the fractional Poincar\'e inequality (see, e.~\!g., Section~2.2 in~\cite{Min03}), we have 
   \begin{equation}\label{s3 lem 3}
   \int_{B_{6r}}|v(\xi)-(v)_{B_{6r}}|^p \, {\rm d}\xi
   \ \leq\ c r^{sp} \int_{B_{6r}}\int_{B_{6r}}|v(\xi)-v(\eta)|^p \, {\rm d}\nu 
   \ \leq\ c r^Q\,.
   \end{equation}
  
  Notice that, in view of the  definitions of $v$ and $\tilde{u}$, we have
    $$
    \big\{v=0\big\} = \big\{\tilde{u} \geq k+d\big\}=\big\{u \geq k\big\}\,.
    $$
    Therefore, the inequality in~\eqref{s3_lem_hyp} yields
    $$
    \big|B_{6r}\cap \big\{v=0\big\}\big| \geq \sigma|B_{6r}|.
    $$
   Also, notice that
   \begin{eqnarray*}
    \log \frac{1}{2 \delta}& = & \frac{1}{|B_{6r}\cap \{v=0\}|}\int_{B_{6r}\cap \{v=0\}}\left(\log\frac{1}{2\delta}-v(\xi)\right) \,{\rm d}\xi\\*[0.8ex]
 &  \leq & \frac{1}{\sigma}\left[\log\frac{1}{2\delta}-(v)_{B_{6r}}\right]\,.
    \end{eqnarray*}
which integrated on~$B_{6r}\cap \{v=\log(1/2\delta)\}$ gives
   \begin{eqnarray*}
    \left|B_{6r}\cap \left\{v = \log \left(\frac{1}{2\delta}\right)\right\}\right|\log \left(\frac{1}{2\delta}\right) & \leq &\frac{1}{\sigma}\int_{B_{6r}}|v(\xi)-(v)_{B_{6r}}| \, {\rm d}\xi\\*[0.8ex]
    & \leq& \frac{1}{\sigma}|B_{6r}|^\frac{1}{p'}\left( \,\int_{B_{6r}}|v(\xi)-(v)_{B_{6r}}|^p \, {\rm d}\xi\right)^\frac{1}{p}\\*[0.8ex]
    & \leq &\frac{c}{\sigma}|B_{6r}|\,,
    \end{eqnarray*}
    where we also used~\eqref{s3 lem 3} and the H\"older inequality.
    
    From all the previous estimates we finally arrive at
    $$
    \big|B_{6r}\cap \big\{\tilde{u}\leq2\delta(k+d)\big\}\big| 
    \ \leq \ \frac{c}{\sigma}\frac{1}{\log\frac{1}{2\delta}}|B_{6r}|\,,
    $$
    so that the desired inequality plainly follows by inserting the definition of~$d$ given in~\eqref{8.0}.
    \end{proof}

\vspace{2mm}
   We are now in the position to refine the estimates in order to prove the main result of this section; i.~\!e.,
   \begin{lemma}\label{s3_lem3}
   Let  $s \in (0,1)$ and $p \in (1,\infty)$ and let $u \in W^{s,p}(\h^n)$ be a weak supersolution to problem~\eqref{problema} such that $u \geq 0$ in $B_R(\xi_0) \subset \Omega$. 

   Let $k \geq 0$ and suppose that there exists~$\sigma \in (0,1]$ such that
   $$
   |B_{6r} \cap \{u \geq k\}| \geq \sigma |B_{6r}|,
   $$
   for some $r$ satisfying $0 < 6r<R$.  Then, there exists a constant~$\delta \in (0,1/4)$ depending on $n,s,p$,~$\sigma$ and $\Lambda$, such that
   \begin{equation}\label{s3_lem3_1}
    \inf_{B_{4r}}u \, \geq\, \delta k -\left(\frac{r}{R}\right)^\frac{sp}{p-1} \T(u_-;\xi_0,R)\,.
    \end{equation}
   \end{lemma}
   \begin{proof}
   We immediately notice that in the case when~$k=0$, the inequality~\eqref{s3_lem3_1} does trivially hold, since $u \geq 0$ in $B_R$. Also, with no loss of generality, we can assume that for $\delta>0$
        \begin{equation}\label{s3_lem3_2}
     \left(\frac{r}{R}\right)^\frac{sp}{p-1} \T(u_-;\xi_0,R)\, \leq\, \delta k.
     \end{equation}
   Now, for any ${r} \leq \rho \leq {6r}$, take a smooth function $\p \in C^\infty_0(B_\rho)$ such that $0 \leq \p \leq 1$, and consider the test function $\psi :=w_-\p^p$, where $w_-:=(\ell-u)_+$, for any $\ell \in (\delta k, 2 \delta k)$.
   Testing Definition~\ref{solution} with such a smooth function~$\psi$ yields
   \begin{eqnarray}\label{s3_lem3_0}
    &&\int_{B_\rho}f(\xi,u)w_-(\xi)\p^p(\xi)\,{\rm d}\xi \nonumber\\*
    &&\quad \leq   \int_{B_\rho}\int_{B_\rho}|u(\xi)-u(\eta)|^{p-2}\big(u(\xi)-u(\eta)\big) \big(w_-(\xi)\p^p(\xi)-w_-(\eta)\p^p(\eta)\big) \,{\rm d}\nu\notag\\*
      &&\qquad + \int_{\h^n \smallsetminus B_\rho} \int_{B_\rho}|u(\xi)-u(\eta)|^{p-2}\big(u(\xi)-u(\eta)\big)w_-(\xi)\p^p(\xi) \,{\rm d}\nu\notag\\*
      &&\qquad  - \int_{B_\rho} \int_{\h^n \smallsetminus B_\rho} |u(\xi)-u(\eta)|^{p-2}\big(u(\xi)-u(\eta)\big)w_-(\eta)\p^p(\eta) \, {\rm d}\nu\notag\\*[0.8ex]
      &&\quad =:  J_1 +J_2 +J_3.
    \end{eqnarray}
   We begin to estimate the term on the left-hand side. As done in our proof of~Lemma~\ref{s3_lem1}, we obtain that
   \begin{eqnarray*}
   	\int_{B_\rho}f(\xi,u)w_-(\xi)\p^p(\xi)\,{\rm d}\xi & =& \int_{B_\rho}(f(\xi,u))_+w_-(\xi)\p^p(\xi){\rm d}\xi -\int_{B_\rho}(f(\xi,u))_-w_-(\xi)\p^p(\xi){\rm d}\xi\\*[0.8ex]
   	& \geq & -\, \ell\|f\|_{L^\infty(B_R)}|B_\rho \cap \{u <\ell\}|.
   \end{eqnarray*}
   Now we focus on the right-hand side of~\eqref{s3_lem3_0}. It is convenient to split~$J_2$ as follows
    \begin{eqnarray*}
    J_2  & =& \int_{\h^n \smallsetminus B_\rho \cap \{u(\eta) < 0\}} \int_{B_\rho}|u(\xi)-u(\eta)|^{p-2}(u(\xi)-u(\eta)) w_-(\xi)\p^p(\xi) \, {\rm d}\nu\\*
         & & \ + \,\int_{\h^n \smallsetminus B_\rho \cap \{u(\eta) \geq 0\}} \          \int_{B_\rho}|u(\xi)-u(\eta)|^{p-2}(u(\xi)-u(\eta))w_-(\xi)\p^p(\xi) \,{\rm d}\nu\\*[0.8ex]
         & =: &  \,J_{2,1} + J_{2,2}.      
   \end{eqnarray*}
   Now, notice that
    \begin{eqnarray*}
     && |\eta^{-1} \circ \xi|_{\h^n}^{-Q-sp}|u(\xi)-u(\eta)|^{p-2}\big(u(\xi)-u(\eta)\big)w_-(\xi)\p^p(\xi)\\[0.0ex]
     && \qquad \qquad \qquad\leq \ \big(\ell + (u(\eta))_-\big)^{p-1}\,\ell \left(\sup_{\xi \in \textup{supp} \, \p}|\eta^{-1} \circ \xi|_{\h^n}^{-Q-sp} \right) \chi_{B_\rho \cap \{u <\ell\}}(\xi)\,,
    \end{eqnarray*}
    which yields
   $$
    J_{2,1}\ \leq\ \ell \left(\sup_{\xi \in \textup{supp} \, \p} \int_{\h^n \smallsetminus B_\rho} (\ell +(u(\eta))_-)^{p-1} |\eta^{-1} \circ \xi|_{\h^n}^{-Q-sp} \, d\eta\right)\big|B_\rho \cap \big\{u < \ell\big\}\big|\,.
    $$
   For~$J_{2,2}$, since $u\geq 0$ in~$B_\rho$, we can write
   \begin{eqnarray*}
    && |\eta^{-1} \circ \xi|_{\h^n}^{-Q-sp}|u(\xi)-u(\eta)|^{p-2}\big(u(\xi)-u(\eta)\big)w_-(\xi)\p^p(\xi)\\*[0.8ex]
    && \qquad\qquad\qquad\qquad\qquad\qquad \leq \ \ell^p \left( \sup_{\xi \in \textup{supp} \, \p} |\eta^{-1} \circ \xi|_{\h^n}^{-Q-sp}\right)\chi_{B_\rho \cap \{u <\ell\}}(\xi)\,.
    \end{eqnarray*}
  By reasoning as above for the integral~$J_3$, we finally arrive at
  $$
   J_2 + J_3 \ \leq\  c \ell \left(\sup_{\xi \in \textup{supp} \, \p}\int_{\h^n \smallsetminus B_\rho} (\ell +(u(\eta))_-)^{p-1} |\eta^{-1} \circ \xi|_{\h^n}^{-Q-sp} \,   {\rm d}\eta\right)|B_\rho \cap \{u < \ell\}|\,.
   $$

It remains to estimate the contribution~$J_1$, and for this one can proceed as seen in the Euclidean setting (see Theorem~1.4 in~\cite{DKP16}); it follows
   \begin{eqnarray*}
   J_1 & \leq& -c\int_{B_\rho}\int_{B_\rho}|w_-(\xi)\p(\xi)-w_-(\eta)\p(\eta))|^p    \, {\rm d}\nu\\*
       & &\ +\, \int_{B_\rho}\int_{B_\rho}\big(\max\big\{w_-(\xi),\, w_-(\eta)\big\}\big)^p  |\p(\xi)-\p(\eta)|^p \,{\rm d}\nu.
   \end{eqnarray*}
   Combining all the above estimates, we get
    \begin{eqnarray}\label{s3_lem3_3}
     && \int_{B_\rho}\int_{B_\rho}|w_-(\xi)\p(\xi)-w_-(\eta)\p(\eta))|^p \, {\rm d}\nu  \notag\\*
          & &\quad \leq c\int_{B_\rho}\int_{B_\rho}(\max\{w_-(\xi),w_-(\eta)\})^p |\p(\xi)-\p(\eta)|^p \, {\rm d}\nu\\*
            &&\qquad  +\, c \ell |B_\rho \cap \{u < \ell\}|\Big(\sup_{\xi \in \textup{supp} \, \p} \int_{\h^n \smallsetminus B_\rho} (\ell +(u(\eta))_-)^{p-1} |\eta^{-1} \circ \xi|_{\h^n}^{-Q-sp} \,  {\rm d}\eta\notag\\*
            && \qquad\qquad\qquad\qquad\qquad\ \ \ \ + \|f\|_{L^\infty(B_R)}\Big) \notag.
    \end{eqnarray}
   At this point an iteration argument is needed. To this aim, define
   $$
    \ell \equiv \ell_j := \delta k + 2^{-j-1}\delta k, \quad
   \rho \equiv \rho_j := {4r} + {2^{1-j}r}, \quad \text{and} \quad  \tilde{\rho}_j := \frac{\rho_{j+1}+\rho_j}{2}, \quad \forall j=0,1,...
   $$
   Note that both $4r<\rho_j, \tilde{\rho}_j <6r$  and
    $$
    \ell_j-\ell_{j+1} \ = \ 2^{-j-2}\delta k \ \geq\  2^{-j-3}\ell_j.
   $$
Moreover, in view of~\eqref{s3_lem3_2}, we have 
   $$
   \ell_0 \ =\ \frac{3}{2}\delta k
   \  \leq \ 2 \delta k - \frac{1}{2}\left(\frac{r}{R}\right)^\frac{sp}{p-1}\!\T(u_-;\xi_0,R)\,,
   $$
   which yields
\begin{equation}\label{12.0r}
   \big\{u < \ell_0\big\} \subset \left\{u < 2 \delta k - \frac{1}{2}\left(\frac{r}{R}\right)^\frac{sp}{p-1}\!\T(u_-;\xi_0,R) \right\}.
\end{equation}
   We are now in the position to apply the estimate in~Lemma~\ref{s3_lem1}; we arrive at
   \begin{equation}\label{s3_lem3_4}
   \frac{ | B_{6r} \cap \{u < \ell_0\} |}{|B_{6r}|} 
   \ \leq \ \frac{\bar{c}}{\sigma \log\frac{1}{2 \delta}}.
    \end{equation}
Now,
   $$
    w_- \equiv w_j = (\ell_j-u)_+
    \ \geq\ (\ell_j-\ell_{j+1}) \chi_{B_\rho \cap \{u <\ell_{j+1}\}}
    \ \geq\ 2^{-j-3}\ell_j \chi_{B_\rho \cap \{u <\ell_{j+1}\}}, \quad \forall j=0,1,...
    $$
   Denote by $B_j := B_{\rho_j}(\xi_0)$ and let $\p_j \in C^\infty_0(B_{\tilde{\rho}_j})$ be such that
    $$
    0 \leq \p_j \leq 1 \qquad \p_j \equiv 1 \mbox{ on } B_{j+1}, \qquad |\nabla_{\h^n} \p_j| \leq 2^{j+3}/r.
    $$
We have
    \begin{eqnarray}\label{s3_lem3_5}
    && (\ell_j-\ell_{j+1})^p \left(\frac{|B_{j+1} \cap \{u < \ell_{j+1}\}|}{|B_{j+1}|}\right)^\frac{p}{p^*}\nonumber\\*
    &&\quad 
    \leq \left( \, \dashint_{B_{j+1}}  w_j^{p^*}\p_j^{p^*} \, {\rm d}\xi \, \right)^ \frac{p}{p^*} \notag\\*[0.8ex]
    &&\quad \leq c\left( \, \dashint_{B_j}  w_j^{p^*}\p_j^{p^*} \, {\rm d}\xi \, \right)^ \frac{p}{p^*} \notag\\*[0.8ex]
    &&\quad \leq  cr^{sp} \dashint_{B_j}\int_{B_j} |w_j(\xi)\p_j(\xi)-w_j(\eta)\p_j(\eta)|^p \, {\rm d}\nu,
    \end{eqnarray}
   where in the last inequality we have used the Sobolev embedding with $p^* = Qp/(Q-sp)$.

    Let us estimate~\eqref{s3_lem3_5} with the aid of~\eqref{s3_lem3_3}. Firstly, by the particular choice of~$\p$, we have 
    \begin{eqnarray*}
    &&c\int_{B_j}\int_{B_j}\big(\max\big\{w_j(\xi),\,w_j(\eta)\big\}\big)^p |\p_j(\xi)-\p_j(\eta)|^p \, {\rm d}\nu\\*
    && \qquad\qquad\qquad\qquad\qquad \leq c 2^{jp}\ell_j^p r^{-p}\int_{B_j}\int_{B_j \cap \{u < \ell_j\}}|\eta^{-1} \circ \xi|_{\h^n}^{-Q-sp+p} \, {\rm d}\xi {\rm d}\eta\\*[0.8ex]
     && \qquad\qquad\qquad\qquad\qquad\leq c 2^{jp}\ell_j^p r^{-sp}|B_j \cap \{u < \ell_j\}|\,,
   \end{eqnarray*}
    where, proceeding as in the proof of the Logarithmic Lemma~1.4 in~\cite{MPPP21}, we have that
     $$
	\int_{B_j} |\eta^{-1} \circ \xi|_{\h^n}^{p-Q-sp} \,{\rm d}\eta \leq c r^{p-sp}.
    $$
    
    Now, notice that, for any $\eta \in \h^n \smallsetminus B_j$ and any $\xi \in \textup{supp} \, \p_j \subset B_{\tilde{\rho}_j}$, it holds
    $$
    \frac{|\eta^{-1} \circ \xi_0|_{\h^n}}{|\eta^{-1} \circ \xi|_{\h^n}}
    \ \leq\ \frac{|\eta^{-1} \circ \xi|_{\h^n}+|\xi^{-1} \circ \xi_0|_{\h^n}}{|\eta^{-1} \circ \xi|_{\h^n}} 
    \ \leq\ c2^{j}.
    $$
    Thus, 
    $$
    \sup_{\xi \in \, \textup{supp} \, \p_j}|\eta^{-1} \circ \xi|_{\h^n}^{-Q-sp}
    \ \leq\ c 2^{j(Q+sp)}|\eta^{-1} \circ \xi_0|_{\h^n}^{-Q-sp}, \quad \forall \eta \in \h^n \smallsetminus B_j\,,
    $$
    which yields
    \begin{eqnarray}\label{13.0}
     && \sup_{\xi \in \, \textup{supp} \, \p_j} \int_{\h^n \smallsetminus B_j} (\ell_j +(u(\eta))_-)^{p-1} |\eta^{-1} \circ \xi|_{\h^n}^{-Q-sp} \, {\rm d}\eta \notag\\*
     & &\qquad\qquad\qquad\leq c 2^{j(Q+sp)}\ell_j^{p-1}r^{-sp} +c2^{j(Q+sp)}r^{-sp} \left(\frac{r}{R}\right)^{sp} [\T(u_-;\xi_0,R)]^{p-1}\notag\\*[0.8ex]
     & &\qquad\qquad\qquad\leq c 2^{j(Q+sp)}r^{-sp}\ell_j^{p-1},
     \end{eqnarray}
   where we have also used the fact that $u \geq 0 $ in $B_R$, the estimate in~\eqref{s3_lem3_2}, and the fact that~$\delta k < \ell_j$.
\vspace{2mm}

   From~\eqref{s3_lem3_3}, \eqref{s3_lem3_5}, and~\eqref{13.0}, we arrive at
   \begin{eqnarray*}
   &&\left(\frac{|B_{j+1} \cap \{u < \ell_{j+1}\}|}{|B_{j+1}|}\right)^\frac{p}{p^*} \\*
   &&\quad\leq\  c2^{j(Q+p+sp)}\frac{\max \{\ell_j^p,\ell_j\}}{ (\ell_j-\ell_{j+1})^p}(1+r^{sp}\|f\|_{L^\infty(B_R)}) \frac{|B_j \cap \{u < \ell_j\}|}{|B_j|}\,.
   \end{eqnarray*}
We are finally in the position to apply the classic iteration Lemma~\ref{giusti}. We set
    \begin{equation*}
    A_j := \frac{|B_j \cap \{u < \ell_j\}|}{|B_j|}
     \end{equation*}
   the previous estimate can be rewritten as follows
    \begin{equation*}
    A_{j+1}^\frac{p}{p^*}\, \leq \, c\frac{2^{j(Q+p+sp)}\max \big\{\ell_j^p,\ell_j\big\}}{\big(\ell_j-\ell_{j+1})^p}(1+r^{sp}\|f\|_{L^\infty(B_R)}\big)A_j\,.
    \end{equation*}
Also, since
   \begin{equation*}
    \frac{\max \{\ell_j^p,\ell_j\}}{(\ell_j-\ell_{j+1})^p} \leq c_p 2^{jp}\max\{1,\frac{3}{2}(\delta k)^{1-p}\} \leq c 2^{jp},
   \end{equation*}
 it follows 
    \begin{equation*}
     A_{j+1} \leq c_1 2^{j(\frac{Qp^*}{p}+2p^*+sp^*)} A_j^{1+\beta}, \qquad \text{with} \ \beta = \frac{sp}{Q-sp},
     \end{equation*}
   and $c_1^{\frac{p}{p^*}}=c\big(1+r^{sp}\|f\|_{L^\infty(B_R)}\big)$.
  Choosing~$\delta>0$ as follows,
    \begin{equation*}
\delta:= \frac{1}{4} \exp \left\{-\frac{\bar{c}c_1^\frac{Q-sp}{sp}2^{(\frac{Q}{p}+s+2)\frac{Q(Q-sp)}{ps^2}}}{\sigma}\right\} < \frac{1}{4}\,,
   \end{equation*}
    and using the estimate in~\eqref{s3_lem3_4}, we arrive at
    \begin{align*}
    A_0 =\frac{|B_{6r}\cap \{u < \ell_0\}|}{|B_{6r}|} 
    \ \leq\ c_1^{-\frac{Q-sp}{sp}}2^{-(\frac{Q}{p}+s+2)\frac{Q(Q-sp)}{ps^2}},
    \end{align*}
which gives
   $$
   \lim_{j \rightarrow \infty}A_j =0,
   $$
 so that $\inf_{B_{4r}}u \geq \delta k$, and hence~\eqref{s3_lem3_1} plainly follows.
 
 \vspace{2mm}
 We consider now the case when~$sp=Q$.  For this, we choose~$ 0 < \varepsilon<s$ and, calling~$s_\varepsilon :=s-\varepsilon$, we have that~$s_\varepsilon p < Q$. Then, for~$p <q< p^*_\varepsilon :=  \frac{Qp}{Q-s_\varepsilon p}$,  we apply the Sobolev inequality in~Theorem~\ref{sobolev}; it follows
 	$$
 	\left(\, \dashint_{B_j} w_j^q \, \p_j^q {\rm d}\xi \right)^\frac{p}{q} \leq |B_j|^{\frac{(p^*_\varepsilon -q)p}{qp^*_\varepsilon} - \frac{p}{q}}  \left(\, \int_{B_j} w_j^{p^*_\varepsilon}\, \p_j^{p^*_\varepsilon} {\rm d}\xi \right)^\frac{p}{p^*_\varepsilon} \leq  c \frac{r^{s_\varepsilon p}}{r_j^Q}  \left(\, \int_{B_j} w_j^{p^*_\varepsilon}\, \p_j^{p^*_\varepsilon} {\rm d}\xi \right)^\frac{p}{p^*_\varepsilon}.
 	$$ 
 	Thus, by the inequality above we get
 	\begin{eqnarray}\label{s3_lem3_6}
 		(\ell_j-\ell_{j+1})^p \left(\frac{|B_{j+1} \cap \{u < \ell_{j+1}\}|}{|B_{j+1}|}\right)^\frac{p}{q} 
 		& \leq & c \left( \, \dashint_{B_j}w_j^q \p_j^q \, {\rm d}\xi\right)^\frac{p}{q}\notag\\*
 		& \leq & c \frac{r^{s_\varepsilon p}}{r_j^Q}  \left(\, \int_{B_j}  w_j^{p^*_\varepsilon} \, \p_j^{p^*_\varepsilon} \, {\rm d}\xi \, \right)^ \frac{p}{p^*_\varepsilon} \notag\\*
 		& \leq & cr^{s_\varepsilon p} \, \dashint_{B_j}\int_{B_j} \frac{|w_j(\xi)\p_j(\xi)-w_j(\eta)\p_j(\eta)|^p}{|\eta^{-1}\circ \xi|_{\h^n}^{Q+s_\varepsilon p}} \, {\rm d}\xi {\rm d}\eta\notag\\*
 		& \leq & cr^{s_\varepsilon p} \left( \,  \dashint_{B_j} w_j^{p} \, \p_j^{p} \, {\rm d}\xi \right.\\*
 		&&  + \, \left. \dashint_{B_j}\int_{B_j} \frac{|w_j(\xi)\p_j(\xi)-w_j(\eta)\p_j(\eta)|^p}{|\eta^{-1}\circ \xi|_{\h^n}^{Q+sp}} \, {\rm d}\xi {\rm d}\eta \right) \, ,\notag
 	\end{eqnarray}
  where  in the last inequality we have split the seminorm~$[w_j\p_j]_{W^{s_\varepsilon,p}}^p$ in the following way 
   \begin{eqnarray*}
    &&\int_{B_j}\int_{B_j} \frac{|w_j(\xi)\p_j(\xi)-w_j(\eta)\p_j(\eta)|^p}{|\eta^{-1}\circ \xi|_{\h^n}^{Q+s_\varepsilon p}} \, {\rm d}\xi {\rm d}\eta\\* &&\qquad\qquad = \, \int_{B_j}\int_{B_j \cap \{|\eta^{-1} \circ \xi|_{\h} \geq 1\}} \frac{|w_j(\xi)\p_j(\xi)-w_j(\eta)\p_j(\eta)|^p}{|\eta^{-1}\circ \xi|_{\h^n}^{Q+s_\varepsilon p}} \, {\rm d}\xi {\rm d}\eta\\*
    &&\qquad\qquad\quad + \int_{B_j}\int_{B_j \cap \{|\eta^{-1}\circ \xi|_{\h^n} <1\}} \frac{|w_j(\xi)\p_j(\xi)-w_j(\eta)\p_j(\eta)|^p}{|\eta^{-1}\circ \xi|_{\h^n}^{Q+s_\varepsilon p}} \, {\rm d}\xi {\rm d}\eta,
 	\end{eqnarray*}
 	and we have estimated the two integrals above as done in~\cite[Proposition~2.8]{MPPP21}.
 	
The first term on the right-hand side of~\eqref{s3_lem3_6} can be treated as follows
 	$$
 	c \, r^{s_\varepsilon p}\, \dashint_{B_j}w_j^p \, {\rm d}\xi \, \leq \,  c \, r^{s_\varepsilon p} \, \ell_j^p \,  \frac{|B_j \cap \{u < \ell_j\}|}{|B_j|}.
 	$$
 	On the other hand, using the same techniques applied in the subcritical case when~$sp <Q$, we have that the second term on the right-hand side in~\eqref{s3_lem3_6} becomes
 	\begin{eqnarray*}
 		&&	cr^{s_\varepsilon p} \,\dashint_{B_j}\int_{B_j}|w_j(\xi)\p_j(\xi)-w_j(\eta)\p_j(\eta)|^p |\eta^{-1}\circ \xi|_{\h^n}^{-Q-sp} \, {\rm d}\xi {\rm d}\eta\notag\\*
 		&&\qquad\quad \leq c \, r^{-\varepsilon p} \, 2^{j(Q+p+sp)} \max\{\ell_j^p,\ell_j\} (1+r^{Q}\, \|f\|_{L^\infty(B_R)}) \frac{|B_j \cap \{u < \ell_j\}|}{|B_j|} \,,
 	\end{eqnarray*}
 	Setting, as before,
 	$$
 	A_j := \frac{|B_j \cap \{u < \ell_j\}|}{|B_j|},
 	$$
 	we get
 	$$
 	A_{j+1} \leq c_1 2^{j \left(\frac{Qq}{p}+2q +sq\right)}A_j^{1+\beta},
 	$$
 	with~$c_1^\frac{p}{q}  := c \, r^{-\varepsilon p}\, ( 1+r^{Q}+r^{Q}\|f\|_{L^\infty(B_R)})$ and~$\beta := \frac{q-p}{p}$.
 	Choosing now
 	$$
 	0 <\delta := \frac{1}{4}\exp \left\{- \frac{\bar{c}c_1^\frac{p}{q-p}2^{\left(\frac{Q}{p}+s+2\right)\frac{qp^2}{(q-p)^2}}}{\sigma}\right\} < \frac{1}{4},
 	$$
 	from~\eqref{s3_lem3_4} we get that
 	\begin{align*}
 		A_0 =\frac{|B_{6r}\cap \{u < \ell_0\}|}{|B_{6r}|} \leq c_1^{-\frac{p}{q-p}}2^{-(\frac{Qq}{p}+2q+sq)\frac{p^2}{(q-p)^2}}.
 	\end{align*}
 	Then, Lemma~\ref{giusti}, with 
 	\begin{equation*}
 		c_0 := c_1 \qquad \mbox{and} \qquad b:= 2^{\frac{Qq}{p}+2q+sq},
 	\end{equation*}
 	yields
 	$$
 	\lim_{j \rightarrow \infty}A_j =0;
 	$$
 	which implies~$\inf_{B_{4r}}u \geq \delta k$ and hence~\eqref{s3_lem3_1} when~$sp=Q$.
	\vspace{2mm}
	  
	  The case when $sp>Q$ can be deduced as in the latter, without relevant modifications, choosing the parameter~$\varepsilon > (s-{Q}/{p})$. With such a choice in hand,   we can use the Sobolev embedding for~$W^{s_\varepsilon,p}$ in Theorem~\ref{sobolev} and the desired result plainly follows as in the previous case when~$sp=Q$.
   \end{proof}

%

\vspace{2mm}
   \subsection{Proof of Theorem~\ref{thm_harnack}}
In order to derive the Harnack inequalities with tail, we firstly need the estimate~\eqref{s3_lem4_1} below, which is a straightforward consequence of the refined positivity expansion proven in the previous section, together with the classical Krylov-Safonov covering lemma (whose proof can be found for instance in~\cite[Lemma~7.2]{KS01}), which can be adjusted to our framework thanks to the role of the nonlocal tail, as shown in the Euclidean framework in the proof of~\cite[Lemma 4.1]{DKP14}.
     \begin{lemma}\label{s3_lem4}
     Let~$s \in (0,1)$,~$p \in (1,\infty)$ and let $u \in W^{s,p}(\h^n)$ be a weak supersolution to~\eqref{problema} such that  $u \geq 0 $ in $B_R \equiv B_R(\xi_0) \subset \Omega$. Then, for any $ B_{6r} \equiv B_{6r}(\xi_0) \subset B_R$, there exist constants~$\e \in  (0,1)$ and~$\textbf{c}=\textbf{c}(n,s,p,\Lambda) \geq 1$ such that
     \begin{equation}\label{s3_lem4_1}
     \left( \, \dashint_{B_{r}} u^\e \,{\rm d}\xi \, \right)^\frac{1}{\e} \leq \textbf{c} \inf_{B_{r}}u + \textbf{c}\left(\frac{r}{R}\right)^\frac{sp}{p-1}\T(u_-;\xi_0,R),
     \end{equation}
      where~$\T(\cdot)$ is defined in~\eqref{tail}.
     \end{lemma}

In the next lemma, we prove that the tail of the positive part of the weak solutions to~\eqref{problema} can be controlled in a precise way.
    \begin{lemma}\label{s3_lem5}
    Let $s \in (0,1)$, $p \in (1,\infty)$ and $u \in W^{s,p}(\h^n)$ be a weak solution to \eqref{problema} such that $u \geq 0$ in $B_R(\xi_0) \subset \Omega$.  Then, for any $0 <r<R$,
    \begin{equation}\label{eq_tail}
    \T(u_+;\xi_0,r) \ \leq \ \textbf{c} \sup_{B_{r}} u + \textbf{c} \left(\frac{r}{R}\right)^\frac{sp}{p-1}\T(u_-;\xi_0,R)
  +\textbf{c}r^\frac{sp}{p-1}\|f\|^\frac{1}{p-1}_{L^\infty(B_R)},
    \end{equation}
	    where $\T(\cdot)$ is defined in~\eqref{tail} and $\textbf{c}=\textbf{c}(n,s,p,\Lambda)$.
    \end{lemma}

    \begin{proof}
    Set $k := \sup_{B_{r}}u$ and choose a cut-off function $\p \in C^\infty_0(B_{r})$ such that $0 \leq \p \leq 1$, $\p \equiv 1 $ on $B_{r/2}$ and $|\nabla_{\h^n} \p | \leq 8/r$. 
Take now the test function $\psi := (u-2k)\p^p$. We have 
   \begin{eqnarray}\label{s3_lem5_1}
  && \int_{B_{r}}f(\xi,u) (u(\xi)-2k)\p^p(\xi) {\rm d} \xi \notag\\*
   & &\qquad =   \int_{B_{r}}\,\int_{B_{r}}|u(\xi)-u(\eta)|^{p-2}\big(u(\xi)-u(\eta)\big)\big((u(\xi)-2k)\p^p(\xi)-(u(\eta)-2k)\p^p(\eta)\big) \, {\rm d}\nu \notag\\*
    & & \qquad \quad + \int_{\h^n \smallsetminus B_{r}}\,\int_{B_{r}}|u(\xi)-u(\eta)|^{p-2}\big(u(\xi)-u(\eta)\big)\big(u(\xi)-2k)\p^p(\xi) \, {\rm d}\nu\notag\\*
    & &\qquad \quad - \int_{B_{r}}\,\int_{\h^n \smallsetminus B_{r}}|u(\xi)-u(\eta)|^{p-2}\big(u(\xi)-u(\eta)\big)\big(u(\eta)-2k)\p^p(\eta\big) \, {\rm d}\nu\notag\\*[0.8ex]
    & & \qquad =:   H_1 + H_2 +H_3.
   \end{eqnarray}
The last two integral in the identity above can be estimated as in the proof of Lemma~4.2 in~\cite{DKP14}; we have
   \begin{eqnarray}\label{s3_lem5_4}
    H_2+H_3 &\ge&
 \ ck|B_{r}|r^{-sp}[\T(u_+;\xi_0,r)]^{p-1} -ck^pr^{-sp}|B_{r}|\\*
 &&-ck|B_{r}|R^{-sp}[\T(u_-;\xi_0,R)]^{p-1}.\notag
   \end{eqnarray}
For what concerns the contribution~$H_1$ in~\eqref{s3_lem5_1}, we have
   \begin{eqnarray}\label{s3_lem5_5}
    H_1 & \geq &-ck^p r^{-p}\int_{B_{r}}\int_{B_{r}}|\eta^{-1}\circ \xi|_{\h^n}^{p-Q-sp} \, {\rm d}\xi {\rm d}\eta \notag\\*[0.8ex]
    & \geq & -ck^pr^{-sp}|B_{r}|,
   \end{eqnarray}
   where  we argued as in the proof of Lemma~1.4 in~\cite{MPPP21}.
   
   The contribution given by the datum~$f$ can be easily estimated as follows,
\begin{equation}\label{17.0}
   \int_{B_{r}}f(\xi,u)(u(\xi)-2k)\p^p(\xi)\, {\rm d} \xi 
   \ \leq\  k|B_{r}|\|f\|_{L^\infty(B_R)}.
\end{equation}

Finally, combining~\eqref{s3_lem5_1} with~\eqref{s3_lem5_4}, \eqref{s3_lem5_5}, and~\eqref{17.0}, we obtain 
   \begin{equation*}
   \T(u_+;\xi_0,r) 
   \ \leq\ ck + c \left(\frac{r}{R}\right)^\frac{sp}{p-1}\T(u_-;\xi_0;R) +cr^\frac{sp}{p-1}\|f\|^\frac{1}{p-1}_{L^\infty(B_R)}\,,
   \end{equation*}
   which gives the desired inequality by recalling the definition of~$k$\,.
   \end{proof}

Armed with the tail estimate in Lemma~\ref{s3_lem5}, and the interpolative inequality given by Theorem~\ref{thm_bdd}, we are ready to complete the proof of the
 Harnack inequality with tail in~\eqref{eq_harnack}. The strategy does generalize that successfully applied in~\cite{DKP14} in the analysis of the homogeneous case in the Euclidean framework.
  \begin{proof}[Proof of  Theorem~{\ref{thm_harnack}}]
Combining the supremum estimate~\eqref{eq_bdd} in Theorem~\ref{thm_bdd} with the tail estimate~\eqref{eq_tail}, we get
  \begin{eqnarray*}
  \sup_{B_{\rho/2}}u & \leq & c\delta^{-\gamma}\left( \, \dashint_{B_{\rho}} u_+^p \, {\rm d}\xi\right)^\frac{1}{p} +c\delta\sup_{B_{\rho}}u\\*
                              & & +\, c\delta\left(\frac{\rho}{R}\right)^\frac{sp}{p-1}\!\T(u_-;\xi_0,R) \,+\, c\delta\rho^\frac{sp}{p-1}\|f\|^\frac{1}{p-1}_{L^\infty(B_R)}.
  \end{eqnarray*}
We now set $\rho:=(\sigma-\sigma')r$, with ${1}/{2} \leq \sigma'<\sigma \leq 1$, so that
  \begin{eqnarray*}
  \sup_{B_{\sigma'r}}u & \leq &
  c\frac{\delta^{-\gamma}}{(\sigma-\sigma')^\frac{Q}{p}}\Big(\sup_{B_{\sigma r}}u\Big)^\frac{p-\e}{p}\left(\, \dashint_{B_{\sigma r}}u^\e \,{\rm d}\xi\right)^\frac{1}{p} +c\delta \sup_{B_{\sigma r}}u\\*  
  & &  +\,c\delta\left(\frac{r}{R}\right)^\frac{sp}{p-1}\T(u_-;\xi_0,R) + c\delta r^\frac{sp}{p-1}\|f\|^\frac{1}{p-1}_{L^\infty(B_R)}\,,
  \end{eqnarray*} 
  where~$\e \in (0,1)$ is the one given by Lemma~\ref{s3_lem4}. We choose the interpolation parameter~$\delta=(4c)^{-1}$, and we obtain
    \begin{eqnarray}\label{strong_hrn_1}
  \sup_{B_{\sigma'r}}u  & \leq & \frac{1}{2}\sup_{B_{\sigma r}}u + \frac{c}{(\sigma-\sigma')^\frac{Q}{\e}} \left(\, \dashint_{B_{r}}u^\e \,{\rm d}\xi\right)^\frac{1}{\e} \\*
  & & +\,c\left(\frac{r}{R}\right)^\frac{sp}{p-1}\T(u_-;\xi_0,R) + cr^\frac{sp}{p-1}\|f\|^\frac{1}{p-1}_{L^\infty(B_R)}\,,\notag
  \end{eqnarray}
  where we also used a suitable Young inequality.
  Note that by Jensen's inequality,  recalling that the exponent~$\e$ given by Lemma~\ref{s3_lem4} is in $(0,1)$, we have that
  $$
 \left(\, \dashint_{B_{ r}}u^\e \,{\rm d}\xi\right)^\frac{1}{\e} = \left(\, \dashint_{B_{ r}}u^{\frac{\e}{p}p} \,{\rm d}\xi\right)^\frac{1}{\e} \leq \left(\, \dashint_{B_{ r}}u^p \,{\rm d}\xi\right)^\frac{1}{p} <\infty,
  $$
  since~$u \in L^p(B_{ r})$. Thus the right-hand term in~\eqref{strong_hrn_1} is finite. Finally, the classic iteration Lemma~\ref{giaquinta_giusti}, with $g(t):=\sup_{B_{ t}}u$, $\tau=\sigma r$, $t:=\sigma'r$, $ \theta := \frac{Q}{\e}$, and $\zeta:=\frac{1}{2}$, yields
  \begin{equation*}
	\sup_{B_{r}}u \ \leq \ c \left(\, \dashint_{B_{ r}}u^\e \,{\rm d}\xi\right)^\frac{1}{\e}\, +\, c\left(\frac{r}{R}\right)^\frac{sp}{p-1}\!\T(u_-;\xi_0,R) \,+\, cr^\frac{sp}{p-1}\|f\|^\frac{1}{p-1}_{L^\infty(B_R)}\,,
  \end{equation*}
which gives the desired inequality~\eqref{eq_harnack} thanks to the result in Lemma~\ref{s3_lem4}.
  \end{proof}

\vspace{2mm}
  \subsection{Proof of Theorem~\ref{thm_weak}}
Let ${1}/{2} < \sigma' < \sigma \leq {3}/{4}$, and let $\p \in C^\infty_0 (B_{{\sigma r}})$ be such that $\p \equiv 1 \, {\rm on } \, B_{\sigma' r}$, and $|\nabla_{\h^n}\p| \leq {4}/{(\sigma-\sigma')r}$.
\vspace{1mm}

We firstly deal with the case when~$sp < Q$. In such a case one can apply Theorem~\ref{sobolev} to the function~$w\p$, with $w := \tilde{u}^\frac{p-q}{p}=(u+d)^\frac{p-q}{p}$, to get
  \begin{equation}\label{hrn_1}
  \left(\, \dashint_{B_{r}}  |w(\xi) \p(\xi)|^{p^*}{\rm d}\xi\right)^\frac{p}{p^*}\ \leq\ c \frac{r^{sp}}{r^Q} \int_{B_{r}} \int_{B_{r}}|w(\xi)\p(\xi)-w(\eta)\p(\eta)|^p {\rm d}\nu.
  \end{equation}
Now, notice that, by the very definition of~$\p$, it follows
  $$
  |\p(\xi)-\p(\eta)|^p \, \leq\, c |\eta^{-1} \circ \xi|^p_{\h^n} \sup_{B_{{\sigma r}}} |\nabla_{\h^n}\p|^p
  \,\leq\, \frac{c}{\big((\sigma-\sigma')r\big)^p} |\eta^{-1} \circ \xi|^p_{\h^n},
  $$
  which yields
  \begin{equation}\label{hrn_2}
  \int_{B_{r}}\int_{B_{r}} \big(\max\{w(\xi),w(\eta)\}\big)^p |\p(\xi)-\p(\eta)|^p {\rm d}\nu\ \leq\ \frac{c r^{-sp}}{(\sigma-\sigma')^p}\int_{B_{{\sigma r}}} w^p(\eta)\, {\rm d}\eta\,,
  \end{equation}
  where we have also used the estimate below
  $$
  \int_{B_{r}} d_{\rm o}(\eta^{-1} \circ \xi)^{-Q-sp+p}{\rm d}\xi\,\leq\, cr^{p-sp}.
  $$ 
 Collecting the estimates~\eqref{hrn_1} and~\eqref{hrn_2} with the Caccioppoli inequality of Theorem~\ref{s3_lem2}, we obtain
 \begin{eqnarray}\label{hrn_3}
 &&\left( \,\dashint_{B_{r}}  |w(\xi) \p(\xi)|^{p^*} {\rm d}\xi \right)^\frac{p}{p^*} \notag\\*
  &&\quad\leq c   \left(\frac{c}{(\sigma-\sigma')^p}+d^{1-p}\left(\frac{r}{R}\right)^{sp}[\T(u_-;\xi_0,R)]^{p-1}\right)\,\dashint_{B_{r}}w^p(\xi)\p^p(\xi)\,{\rm d}\xi\\
  &&\qquad   +\, cd^{1-q}r^{sp} \|f\|_{L^\infty(B_R)}\,,\notag
  \end{eqnarray}
  where we have used the fact that 
  \begin{equation*}
  \sup_{\xi \in \textup{supp} \, \p} \, \int_{\h^n \smallsetminus B_{r}} d_{\rm o}(\eta^{-1} \circ \xi)^{-Q-sp} {\rm d}\eta \leq c r^{-sp}.
  \end{equation*}

Now, we choose $d$ as in the proof of Lemma~\ref{s3_lem1}; see~\eqref{8.0} there. It follows
  \begin{equation}\label{19.0}
  \left(\  \dashint_{B_{{\sigma' r}}}  \tilde{u}^{(p-q)\frac{Q}{Q-sp}}\,{\rm d}\xi\right)^\frac{Q-sp}{Q}
  \ \leq\  \frac{c}{(\sigma-\sigma')^p}\, \dashint_{B_{\sigma r}} \tilde{u}^{p-q}\, {\rm d}\xi + cr^{sp} \|f\|_{L^\infty(B_R)}\,,
  \end{equation}
  where $c$ depends only on $n$, $p$, $s$ and the structural constant $\Lambda$ defined in \eqref{def_lambda}.
  
 Let $t = (p-q){Q}/({Q-sp})$ for any $q \in (1,p)$. 
 Thanks to a standard finite Moser iteration, the inequality in~\eqref{19.0} becomes
  \begin{eqnarray}\label{hrn_4}
   \left( \, \dashint_{B_{\frac{r}{2}}}  u^t{\rm d}\xi\right)^\frac{1}{t} \ \leq\  c \left( \,\dashint_{B_\frac{3 r}{4}} \tilde{u}^{t'} {\rm d}\xi\right)^\frac{1}{t'}  + \, cr^\frac{Qsp}{t(Q-sp)}\|f\|_{L^\infty(B_R)}^\frac{Q}{t(Q-sp)}\,,
  \end{eqnarray}
for any  $0<t'<t<{(p-1)Q}/(Q-sp)$. 
%
%
  Since ${6r}< R$, we can apply Lemma~\ref{s3_lem4} with $\e=t'$ there; it follows
$$
  	\left( \, \dashint_{B_{\frac{r}{2}}}  u^t{\rm d}\xi\right)^\frac{1}{t} \ \leq\  c\inf_{B_{\frac{3r}{4}}} u +c \left(\frac{r}{R}\right)^\frac{sp}{p-1} \T(u_-;\xi_0,R)  + \, cr^\frac{Qsp}{t(Q-sp)}\|f\|_{L^\infty(B_R)}^\frac{Q}{t(Q-sp)}\,,
$$
which provides the desired inequality, up to relabelling~$r$. 
  
  \vspace{2mm}
	We investigate now the case when~$sp=Q$.   Fix~$ 0 < \varepsilon < s$, and set~$s_\varepsilon := s -\varepsilon$. Since~$s_\varepsilon p < sp=Q$, we can make use of the Sobolev inequality for the function~$w \p$, to get
	\begin{eqnarray}\label{hrn_5}
	&&	\left( \, \dashint_{B_r} |w(\xi)\p(\xi)|^{p^*_\varepsilon} \, {\rm d}\xi\right)^\frac{p}{p^*_\varepsilon} \nonumber\\*
 && \leq  c \, \frac{r^{s_\varepsilon p}}{r^Q} \int_{B_r}\int_{B_r}|w(\xi)\p(\xi)-w(\eta)\p(\eta)|^p |\eta^{-1}\circ \xi|_{\h^n}^{-Q-s_\varepsilon p} \, {\rm d}\xi {\rm d}\eta\notag\\*
		& \leq & c \, r^{s_\varepsilon p} \left( \quad \dashint_{B_r} |w(\xi)\p(\xi)|^p \, {\rm d}\xi \right.\\*
		&& + \left. \,  \dashint_{B_r}\int_{B_r}|w(\xi)\p(\xi)-w(\eta)\p(\eta)|^p |\eta^{-1}\circ \xi|_{\h^n}^{-Q-sp} \, {\rm d}\xi {\rm d}\eta \right) \,,\notag
	\end{eqnarray}
    where we have obtained the last inequality by following by the same argument used at the end of the proof of Lemma~\ref{s3_lem3}. Now, we choose~$d$ as in~\eqref{8.0}, so that the inequality in~\eqref{hrn_5} becomes
	\begin{eqnarray*}		
		&&\left( \, \dashint_{B_r} |w(\xi)\p(\xi) |^{p^*_\varepsilon} \, {\rm d}\xi\right)^\frac{p}{p^*_\varepsilon}\notag\\*
		&&\quad \leq  c\, r^{-\varepsilon p} \left( \frac{c}{(\sigma-\sigma')^p} +d^{1-p}\left(\frac{r}{R}\right)^{sp} [\T(u_-;\xi_0,R)]^{p-1}+ r^Q\right) \, \dashint_{B_r} |(w\p)(\xi)|^p \, {\rm d}\xi\\*
		&&\quad\quad + \, c d^{1-q}r^{(s-\varepsilon) p} \|f\|_{L^\infty(B_R)}\\*[0.8ex]
		&&\quad \leq \frac{c}{(\sigma-\sigma')^p}  \,  \dashint_{B_r} |(w\p)(\xi)|^p \, {\rm d}\xi + \, c r^{(s-\varepsilon)p} \|f\|_{L^\infty(B_R)} \,.
	\end{eqnarray*}
	Thus, recalling the definition of~$w$ and that of the cut-off function~$\p$, we have that
	\begin{equation*}
		\left( \, \dashint_{B_{\sigma' r}} \tilde{u}^{(p-q)\frac{s}{\varepsilon}} \, {\rm d}\xi\right)^\frac{\varepsilon}{s} \leq \frac{c}{(\sigma-\sigma')^p} \, \dashint_{B_{\sigma r}} \tilde{u}^{p-q} \, {\rm d}\xi + c \, r^{(s-\varepsilon)p}\|f\|_{L^\infty(B_R)} \,.
	\end{equation*}
	Set~$t = (p-q){s}/{\varepsilon}$, for any~$q \in (1,p)$;  a standard application of the finite Moser iteration yields
	\begin{eqnarray*}
		\left( \, \dashint_{B_{\frac{r}{2}}}  u^t{\rm d}\xi\right)^\frac{1}{t} \ \leq\  c \left( \,\dashint_{B_\frac{3 r}{4}} \tilde{u}^{t'} {\rm d}\xi\right)^\frac{1}{t'}  + \, cr^\frac{Q(s-\varepsilon)}{t \varepsilon}\|f\|_{L^\infty(B_R)}^\frac{s}{t \varepsilon}\,,
	\end{eqnarray*}
	for any  $0<t'<t<{(p-1)s}/{\varepsilon}$. 
	\vspace{2mm}
	
	Finally, we can conclude as in the proof in the case when~$sp <Q$; that is, it suffices to apply Lemma~\ref{s3_lem4}, with $\e=t'$ there, in order to get 
	$$
	\left( \, \dashint_{B_{\frac{r}{2}}}  u^t{\rm d}\xi\right)^\frac{1}{t} \ \leq\  c\inf_{B_{\frac{3r}{4}}} u +c \left(\frac{r}{R}\right)^\frac{Q}{p-1} \T(u_-;\xi_0,R)  + \, c \, r^\frac{Q(s-\varepsilon)}{t \varepsilon}\|f\|_{L^\infty(B_R)}^\frac{s}{t \varepsilon  }\,,
	$$
	which provides the desired inequality, up to relabelling~$r$. 
  \vspace{2mm}
  
  The case when $sp>Q$ can be deduced as in the latter, without relevant modifications, choosing the parameter~$\varepsilon > (s-{Q}/{p})$. With such a choice in hand, it plainly follows that~$s_\varepsilon p < Q$, and thus one can use the Sobolev embedding in Theorem~\ref{sobolev} and proceed as done in the limit case when~$sp=Q$.\hfill$\Box$

\vspace{3mm}
  \section{The fractional subLaplacian case}\label{sec_limit}
  In this section we focus our attention on the case when~$p=2$ in the particular situation in which the operator~$\l$ defined in~\eqref{operatore} does coincide with the fractional subLaplacian~$(-\Delta_{\h^n})^s$ on~$\h^n$, so that problem~\eqref{problema} does reduce to
      \begin{eqnarray}\label{fractional_pbm}
           \begin{cases}
   (-\Delta_{\h^n})^s u = 0  & \text{in} \ \Omega\subset\h^n,\\[0.4ex]
     u = g  & \text{in}\ \h^n \smallsetminus \Omega,
          \end{cases}
               \end{eqnarray}
   where $g \in H^s(\h^n)\equiv W^{s,2}(\h^n)$.
   \vspace{1mm}
   
   We now recall the precise definition of the fractional subLaplacian operator; that is, 
   \begin{equation}\label{frac_lap}
	(-\Delta_{\h^n})^s u(\xi) = C(n,s)\ P.~\!V.\int_{\h^n} \frac{u(\xi)-u(\eta)}{|\eta^{-1} \circ \xi |_{\h^n}^{Q+2s}}\, {\rm d}\eta, \qquad \forall \xi \in \h^n,
    \end{equation}
where~$|\cdot|_{\h^n}$ is the norm defined in~\eqref{korany_folland}, and~$C(n,s)$ is given by
    \begin{equation}\label{C(n,s)}
    	C(n,s) = \frac{c_1(n,s)\omega_{2n}}{n c_2(n,s)},
    \end{equation} 
    with
    \begin{equation}\label{c(n,s)}
    	c_1(n,s) = \left(\, \int_{\r^{2n+1}} \frac{1-\cos(x_1)}{\|\eta\|^{1+2(n+s)}}\,{\rm d}\eta\right)^{-1} \quad \text{and} \quad
    	c_2(n,s) = \int_{\partial B_1}\frac{x_1^2}{|\eta|_{\h^n}^{Q+2s}}\,{\rm d}\sigma(\eta)\,,
    \end{equation}
    for  $\eta := (x_1,\cdots,x_{2n},t)$.
In the display above, we denote by $\|\cdot\|$ the standard Euclidean norm on~$\r^{2n+1}$, and  by $\sigma$ the surface measure on~$\partial B_1$; see, e.~\!g., Proposition~1.15 in~\cite{FS82}.
   
   \vspace{2mm}
    \subsection{Asymptotics of the fractional subLaplacian}
%
%
%
    \begin{proof}[Proof of Proposition~{\rm \ref{limit_of_fraclap}}]
For the sake of readability, we denote the points~$\xi$ in $\h^n$ as follows,
	$$
	\xi := (x_1,\dots,x_{2n},t)\,.
	$$
	Also, it is convenient to use the weighted second order integral definition of the fractional sublaplacian,
			$$
			(-\Delta_{\h^n})^s u (\xi) \, =\, -\frac{1}{2}C(n,s) \int_{\h^n}\frac{u(\xi \circ \eta)+ u(\xi \circ \eta^{-1} )-2u(\xi)}{|\eta|_{\h^n}^{Q+2s}}\,{\rm d}\eta, \quad \forall \xi \in \h^n\,;
			$$
			see, e.~\!g., \cite[Proposition~1.4]{FMPPS18} and \cite[Proposition~3.2]{DPV12}.
We also recall that,  given $D^{2,*}_{\h^n} u (\xi) $ in~\eqref{symm_hess}, one has
	$$
	\Delta_{\h^n} u \equiv \textup{Tr}(D^{2,*}_{\h^n}u) = \sum_{i=1}^{2n} X_i^2 u\,.
	$$	

As the computation below shows, we have no contribution outside the unit ball in the limit as $s$ goes to $1^-$,
	\begin{eqnarray*}
		\left|\quad \int_{\h^n \smallsetminus B_1}\frac{u(\xi \circ \eta)+ u(\xi \circ \eta^{-1} )-2u(\xi)}{|\eta|_{\h^n}^{Q+2s}} \, {\rm d}\eta \quad \right| & \leq  & 4\|u\|_{L^\infty(\h^n)}  \int_{\h^n \smallsetminus B_1}\frac{1}{|\eta|_{\h^n}^{Q+2s}} \, {\rm d}\eta\\*[0.8ex]
		& \leq & 4c\|u\|_{L^\infty(\h^n)}.
	\end{eqnarray*}
	Hence, recalling~\eqref{C(n,s)} and~\eqref{limit_c(n,s)}, it follows
	\begin{equation}
		\lim_{s \rightarrow 1^-}-\frac{C(n,s)}{2}\int_{\h^n \smallsetminus B_1(0)}\frac{u(\xi \circ \eta)+ u(\xi \circ \eta^{-1} )-2u(\xi)}{|\eta|_{\h^n}^{Q+2s}} \, {\rm d}\eta\,=\,0.
	\end{equation}
	It remain to estimate the integral contribution in the unit ball. In view of Proposition~\ref{prop_taylor}, for any $\eta =(x,t)$, one get
	\begin{equation}\label{ss3_e4}
		u(\xi \circ \eta^{-1})= P_2(u,\xi)(\xi \circ \eta^{-1}) + {\text{o}}(|\eta|_{\h^n}^3) \quad \text{as} \ |\eta|_{\h^n}\to 0\,,
	\end{equation}
	where $P_2(u,\xi)$ is the Taylor polynomial of $\h^n$-degree~2 associated to~$u$ and centered at~$\xi$ presented in Section~\ref{sec_heis}. 
	
	Also, by the very definition of Taylor polynomial, it follows
	$$
	P_2(u,\xi)(\xi \circ \eta^{-1})\, =\, P_2\big(u(\xi \circ \cdot),0\big)(\eta^{-1})  \, = \, u(\xi)- \big(\nabla_{\h^n}u(\xi),\,\partial_t u(\xi)\big)\cdot \eta \, + \, \frac{1}{2}\langle x,D^{2,*}_{\h^n}u(\xi) \cdot x\rangle.
	$$
Thus, inequality~\eqref{ss3_e4} yields
	\begin{equation*}
		u(\xi \circ \eta^{-1}) =  u(\xi)- \big(\nabla_{\h^n}u(\xi),\,\partial_t u(\xi)\big)\cdot \eta + \frac{1}{2}\langle x,D^{2,*}_{\h^n}u(\xi) \cdot x\rangle + {\text{o}}(|\eta|_{\h^n}^3)\quad \text{as} \ |\eta|_{\h^n}\to 0\,.
	\end{equation*}
Using again the result in~Proposition~\ref{prop_taylor}, we arrive at
	\begin{eqnarray*}
		&& \left|\quad \int_{B_1}\frac{u(\xi \circ \eta)+ u(\xi \circ \eta^{-1} )-2u(\xi)-\langle x ,D^{2,*}_{\h^n}u(\xi) \cdot x \rangle}{|\eta|_{\h^n}^{Q+2s}} \, {\rm d}\eta \quad \right|\notag\\*[0.8ex]
		&& \qquad\qquad\qquad\qquad\qquad\quad \leq \int_{B_1}\frac{|u(\xi \circ \eta) -P_2(u(\xi \circ \cdot),0)(\eta)|+{\text{o}}(|\eta|_{\h^n}^3)}{|\eta|_{\h^n}^{Q+2s}} \, {\rm d}\eta\notag\\*[0.8ex]
		&& \qquad\qquad\qquad\qquad\qquad\quad \leq \int_{B_1}\frac{{\text{o}}(|\eta|_{\h^n}^3)}{|\eta|^{Q+2s}_{\h^n}} {\rm d}\eta\notag\\*[0.8ex]
		&& \qquad\qquad\qquad\qquad\qquad\quad \leq \int_{B_1}\frac{1}{|\eta|_{\h^n}^{Q-1}}\  =:\  c(n).
	\end{eqnarray*}
	The preceeding  estimate yields
	\begin{eqnarray}\label{ss3_e5}
		&& \lim_{s \rightarrow 1^-}-\frac{C(n,s)}{2}\int_{B_1}\frac{u(\xi \circ \eta)+ u(\xi \circ \eta^{-1} )-2u(\xi)}{|\eta|_{\h^n}^{Q+2s}} \, {\rm d}\eta \notag\\
		&& \qquad \qquad \qquad \qquad \qquad \qquad = \lim_{s \rightarrow 1^-} -\frac{C(n,s)}{2}\int_{B_1}\frac{\langle x,D^{2,*}_{\h^n}u(\xi) \cdot x \rangle }{|\eta|_{\h^n}^{Q+2s}}\,{\rm d}\eta.
	\end{eqnarray}
	
	Now, notice that for any~$i \neq j$ it holds
	\begin{eqnarray*}
		&&\int_{B_1}\left(\frac{1}{2}(X_iX_ju(\xi)+X_jX_iu(\xi))\right) x_i \,  \cdot \, x_j \,{\rm d}\eta\\*
  &&\quad
		= - \int_{B_1}\left(\frac{1}{2}(X_iX_ju(\xi)+X_jX_iu(\xi))\right) \tilde{x}_i \, \cdot \, \tilde{x}_j\, {\rm d}\tilde{\eta},
	\end{eqnarray*}
	where $\tilde{x}_{i} =x_i$, for $i \neq j$, and $\tilde{x}_j =- x_j$. Therefore,
	\begin{equation}\label{ss3 _e6}
		\int_{B_1}\left(\frac{1}{2}(X_iX_ju(\xi)+X_jX_iu(\xi))\right) x_i \,  \cdot \, x_j {\rm d}\eta =0, \qquad \text{for}\ i \neq j.
	\end{equation}
	Moreover, for any fixed index~$i$, making using of the polar coordinates, namely Proposition~1.15 in~\cite{FS82}, we get that there exists a unique Borel measure~$\sigma$ on $B_1$ such that, up to permutations,
	\begin{eqnarray*}
		\int_{B_1}\frac{X_i^2 u(\xi) x_i^2}{|\eta|_{\h^n}^{Q+2s}}\, {\rm d}\eta & = & X_i^2 u(\xi) \int_{0}^{1}\int_{\partial B_1}\frac{x_1^2r^{Q+1}}{|{\Phi}_{r}(\eta)|_{\h^n}^{Q+2s}}\, {\rm d}\sigma(\eta)\,{\rm d}r\notag\\*[0.8ex]
		& = & X_i^2 u(\xi)\int_{\partial B_1}\frac{x_1^2}{|\eta|_{\h^n}^{Q+2s}}\,{\rm d}\sigma(\eta)\int_{0}^{1}\frac{1}{r^{2s-1}}\, {\rm d}r\notag\\*[0.8ex]
		& =& \frac{c_2(n,s)}{2(1-s)}X_i^2u(\xi).
	\end{eqnarray*}
	where~$c_2(n,s)$ is defined in~\eqref{c(n,s)}.
	\vspace{1mm}
	
	Hence, recalling the result in~\cite[Corollary~4.2]{DPV12} which shows that
	\begin{equation}\label{limit_c(n,s)}
		\lim_{s \rightarrow 1^-}\frac{c_1(n,s)}{s(1-s)}= \frac{4n}{\omega_{2n}},
	\end{equation}
	where~$\omega_{2n}$ denotes the~$(2n)$-dimensional Lebesgue measure of the unit sphere~$\mathbb{S}^{2n}$, we finally obtain that
	\begin{eqnarray*}
		\lim_{s \rightarrow 1^-}(-\Delta_{\h^n})^s u(\xi)  & = &\lim_{s \rightarrow 1^-} -\frac{C(n,s)}{2}\int_{B_1} \frac{\langle x,D^{2,*}_{\h^n}u(\xi)  \cdot x\rangle }{|\eta|_{\h^n}^{Q+2s}}\,{\rm d}\eta\\*
		& =& \lim_{s \rightarrow 1^-} -\frac{C(n,s)}{2} \sum_{i=1}^{2n}\int_{B_1}\frac{X_i^2 u(\xi) x_i^2}{|\eta|_{\h^n}^{Q+2s}}\, {\rm d}\eta\\*
		& =& \lim_{s \rightarrow 1^-} -\frac{c_1(n,s)\omega_{2n}}{4n(1-s)} \sum_{i=1}^{2n}X_i^2 u(\xi) \ \equiv \  -\Delta_{\h^n}u(\xi)\,,
	\end{eqnarray*}
	as desired.
    \end{proof}

\vspace{0mm}
    \subsection{Robustness of the nonlocal Harnack estimates}
    
The proofs of Theorem~\ref{thm_harnack2} and Theorem~\ref{thm_weakharnack2} can be plainly deduced from the ones in Section~\ref{sec_weak} for the subcritical case, by taking there $p=2$, $f\equiv0$, and the Korani-Folland norm in place of the generic  homogeneous norm~$d_{\rm o}$. Below we stated the related needed lemmata, by indicating only the modifications in the estimates where a special care on the involved quantities is needed in order to successfully obtaining the desired robustness in the limit as $s$ goes to $1$.
  \vspace{3mm}

  Firstly, we need the related positivity expansion, which can be condensed in the following two lemmata.
    \begin{lemma}\label{s4_lem1}
   	Let $u \in H^s(\h^n)$, with $s \in (0,1)$, be a weak supersolution to problem~\eqref{fractional_pbm} such that $u \geq 0$ in $B_R(\xi_0) \subset \Omega$. Let $k \geq 0$. Suppose that there exists~$\sigma \in (0,1]$ such that
   	\begin{equation}\label{s4_lem1_hyp1}
   		|B_{6r} \cap \{u \geq k\}| \,\geq\, \sigma|B_{6r}|,
   	\end{equation}
   	for some $r>0$ such that $B_{8r} \equiv B_{8r}(\xi_0) \subset B_R(\xi_0)$. Then there exists a constant $\bar{\textbf{c}} \equiv \bar{\textbf{c}}\,(n)$ such that
   	\begin{equation}
   		\left|B_{6r} \cap \left\{u \leq 2\delta k - \frac{1-s}{2}\left(\frac{r}{R}\right)^{2s}\T(u_-;\xi_0,R)\right\}\right| \,\leq\, \frac{\bar{\textbf{c}}}{\sigma \log \frac{1}{2\delta}}|B_{6r}|
   	\end{equation}
   	holds for all $\delta \in (0,1/4)$, where $\T(\cdot)$ is defined in~\eqref{tail} taking $p=2$ there.
   \end{lemma}
   
   \begin{proof}
 It suffices to repeat the proof of Lemma~\ref{s3_lem1} by choosing the parameter~$d$ in~\eqref{8.0} as follows,
%
   	$$
   	d := \frac{1-s}{2}\left(\frac{r}{R}\right)^{2s}\T(u_-;\xi_0,R)\,. 
   	$$
   \end{proof}

    \begin{lemma}\label{s4_lem3}
   	Let  $s \in (0,1)$ and let $u \in H^s(\h^n)$ be a weak supersolution to problem~\eqref{fractional_pbm} such that $u \geq 0$ in $B_R(\xi_0) \subset \Omega$. 
   	
   	Let $k \geq 0$ and suppose that there exists~$\sigma \in (0,1]$ such that
   	$$
   	|B_{6r} \cap \{u \geq k\}| \geq \sigma |B_{6r}|,
   	$$
   	for some $r$ satisfying $0 < 6r<R$. Then, there exists a constant~$\delta \in (0,1/4)$ depending on $n$ and~$\sigma$ for which
   	\begin{equation}\label{s4_lem3_1}
   		\inf_{B_{4r}}u \,\geq\, \delta k -(1-s)\left(\frac{r}{R}\right)^{2s}\T(u_-;\xi_0,R)\,.
   	\end{equation}
   \end{lemma}
   \begin{proof}
   The proof is basically contained in that of Lemma~\ref{s3_lem3}. It suffices to replace formula~\eqref{s3_lem3_2} with
   	\begin{equation}\label{s4_lem3_2}
   	(1-s)\left(\frac{r}{R}\right)^{2s} \T(u_-;\xi_0,R) \, \leq \, \delta k\,,
   	\end{equation}
so that the same iterative process will give
   	$$
   	\ell_0 = \frac{3}{2}\delta k \, \leq\, 2 \delta k - \frac{1-s}{2}\left(\frac{r}{R}\right)^{2s}\T(u_-;\xi_0,R)\,,
   	$$
   	which in turn yields the following estimate (in place of~\eqref{12.0r} in Lemma~\ref{s3_lem3})
   	$$
   	\{u < \ell_0\} \subset \left\{u < 2 \delta k - \frac{1-s}{2}\left(\frac{r}{R}\right)^{2s}\T(u_-;\xi_0,R) \right\}\,.
   	$$
The proof will then follow with no further modifications at all.
    \end{proof}
    \vspace{2mm}
    
    As well as in the proof in the general nonlinear framework presented in Section~\ref{sec_weak}, we can obtain for the pure subLaplacian case the analogue of the estimate in~Lemma~\ref{s3_lem4} and that of the Tail control estimate stated in~Lemma~\ref{s3_lem5}. For the sake of the reader, we prefer to restate these results by stressing the novelty of the dependance on~$s$ here. No modifications in the related proofs are essentially needed, thanks to the results obtained in Lemma~\ref{s4_lem1} and Lemma~\ref{s4_lem3}.
    \begin{lemma}\label{s4_lem4}
  	Let $u \in H^s(\h^n)$ be a weak supersolution to~\eqref{fractional_pbm} such that  $u \geq 0 $ in $B_R \equiv B_R(\xi_0) \subset \Omega$. Then, for any $ B_{6r} \equiv B_{6r}(\xi_0) \subset B_R$, there exist constants~$\e \in  (0,1)$  such that
  	\begin{equation*}
  		\left( \, \dashint_{B_{r}} u^\e \,{\rm d}\xi \, \right)^\frac{1}{\e} \leq \textbf{c} \inf_{B_{r}}u + \textbf{c}\,(1-s)\left(\frac{r}{R}\right)^{2s}\T(u_-;\xi_0,R).
  	\end{equation*}
	   	where $\textbf{c}$ depends only on $n$.
    \end{lemma}
  
    \begin{lemma}\label{s4_lem5}
   	Let $s \in (0,1)$ and $u \in H^s(\h^n)$ be a weak solution to \eqref{fractional_pbm} such that $u \geq 0$ in $B_R(\xi_0) \subset \Omega$.  Then, for any $0 <r<R$,
   	\begin{equation*}
   		\T(u_+;\xi_0,r) \leq \textbf{c} \sup_{B_{r}} u + \textbf{c}\, (1-s)\left(\frac{r}{R}\right)^{2s}\T(u_-;\xi_0,R)\,,
   	\end{equation*}
   	where $\textbf{c}$ depends only on $n$.
   \end{lemma}

\vspace{0.1mm}

\end{document}